\documentclass[a4paper, reqno]{amsart}

\usepackage[utf8]{inputenc}

\usepackage{amssymb}
\usepackage{eucal, stmaryrd}
\usepackage{mathtools}
\usepackage{graphicx}
\usepackage{array}
\usepackage{hyphenat}
\usepackage[table]{xcolor}
\usepackage{cite}
\usepackage[pdftex,colorlinks,citecolor=black,linkcolor=black,urlcolor=black,bookmarks=false]{hyperref}

\usepackage{multirow}

\newtheorem{lemma}{Lemma}[section]
\newtheorem{theorem}[lemma]{Theorem}

\makeatletter
\newcommand*\smallbullet{\mathpalette\smallbullet@{0.7}}
\newcommand*\smallbullet@[2]{\mathbin{\vcenter{\hbox{\scalebox{#2}{$\m@th#1\bullet$}}}}}
\makeatother

\newcommand{\sub}[2]{\mathrm{Sub}(#1, #2)}
\newcommand{\inds}[1]{\mathcal{I}_{#1}}

\newcommand{\set}[1]{\llbracket #1\rrbracket}

\DeclareMathOperator{\Span}{Span}
\DeclareMathOperator{\Aff}{Aff}

\newcommand{\tp}{{\sf T}}
\DeclareMathOperator{\trace}{tr}
\newcommand{\mc}[1]{\mathcal{#1}}

\newcommand{\ortho}{\mathrm{O}}         
\newcommand{\stab}[1]{{\rm Stab}(#1)}  
\newcommand{\smallpmatrix}[1]{\left(\begin{smallmatrix}#1\end{smallmatrix}\right)}  

\newcommand{\1}{{\bf 1}}
\DeclareMathOperator{\rank}{rank}

\newcommand{\R}{\mathbb{R}}

\newcommand{\Z}{\mathbb{Z}}

\newcommand{\floor}[1]{\left\lfloor #1\right\rfloor}

\newcommand{\kpbn}{\Delta}
\newcommand{\lassn}{\Lambda} 

\newcommand{\MS}[1]{{\rm MS}(#1)}
\newcommand{\spindle}[2]{{\rm S}(#1, #2)}
\newcommand{\comp}[1]{{\rm K}_{#1}}

\newenvironment{optprob}
{
  \arraycolsep=0pt
  \begin{array}{r@{\ }l@{\quad}l}
}%
{
  \end{array}
}

\newcommand{\onerow}[1]{\multicolumn{2}{l}{#1}}

\newcolumntype{C}{>{{}}c<{{}}} 

{
    \end{array}
    \right.
}

\newlength\claimlen%
\newcommand{\assert}[1]{%
  \begin{minipage}{\claimlen}
    #1
  \end{minipage}%
}

\newcommand{\defi}[1]{\textit{#1}}

 \makeatletter
 \newcommand*{\centerfloat}{%
   \parindent\z@%
   \leftskip\z@\@plus 1fil\@minus\textwidth%
   \rightskip\leftskip%
   \parfillskip\z@skip}
 \makeatother

\title{Obtuse almost-equiangular sets}

\author{Christine Bachoc}
\address{C. Bachoc, Institut de Mathématiques de Bordeaux, UMR 5251, université de Bordeaux, 351 cours de la Libération, 33400 Talence, France.}
\email{Christine.Bachoc@math.u-bordeaux.fr}

\author{Bram Bekker}
\address{A.J.F. Bekker, Delft Institute of Applied Mathematics,
  Delft University of Technology, Mekelweg~4, 2628~CD Delft, The
  Netherlands.}
\email{B.Bekker@tudelft.nl}

\author{Philippe Moustrou}
\address{P. Moustrou, Institut de Mathématiques de Toulouse, Université Toulouse Jean Jaurès, 5 Allée Antonio Machado, 31058 Toulouse, France.}
\email{philippe.moustrou@math.univ-toulouse.fr}

\author{Fernando Mário de Oliveira Filho}
\address{F.M. de Oliveira Filho, Delft Institute of Applied Mathematics,
  Delft University of Technology, Mekelweg~4, 2628~CD Delft, The
  Netherlands.}
\email{F.M.deOliveiraFilho@tudelft.nl}

\thanks{The second author is supported by the grant OCENW.KLEIN.024 of the Dutch
  Research Council (NWO)}


\date{April 15, 2025}

\begin{document}
\begin{abstract}
For~$t \in [-1, 1)$, a set of points on the $(n-1)$-dimensional unit sphere is called \defi{$t$-almost equiangular} if among any three distinct points there is a pair with inner product~$t$.  We propose a semidefinite programming upper bound for the maximum cardinality~$\alpha(n, t)$ of such a set based on an extension of the Lovász theta number to hypergraphs.  This bound is at least as good as previously known bounds and for many values of~$n$ and~$t$ it is better.

We also refine existing spectral methods to show that~$\alpha(n, t) \leq 2(n+1)$ for all~$n$ and~$t \leq 0$, with equality only at~$t = -1/n$.  This allows us to show the uniqueness of the optimal construction at~$t = -1/n$ for~$n \leq 5$ and to enumerate all possible constructions for~$n \leq 3$ and~$t \leq 0$.
\end{abstract}

\maketitle
\markboth{C. Bachoc, A.J.F. Bekker, P. Moustrou and F.M. de Oliveira Filho}{Obtuse almost-equiangular sets}

\setcounter{tocdepth}{1}
\tableofcontents


\section{Introduction}%
\label{sec:intro}

For~$x$, $y \in \R^n$, denote by~$x\cdot y$ the Euclidean inner product between~$x$ and~$y$.  For integer~$n \geq 2$, let~$S^{n-1} = \{\, x \in \R^n : \|x\|=1\,\}$ be the $(n-1)$-dimensional unit sphere.  
Given an inner product~$t \in [-1, 1)$, a set~$S \subseteq S^{n-1}$ is \defi{$t$-almost-equiangular} if every $3$-subset~$\{x, y, z\}$ of~$S$ is such that~$t \in \{x\cdot y, x\cdot z, y\cdot z\}$.  In the literature, the word ``almost'' is often replaced by ``nearly''.  An \defi{obtuse almost-equiangular set} is a $t$-almost equiangular set with~$t \leq 0$.  A $0$-almost-equiangular set is also often called \defi{almost-orthogonal}.  Similarly, one may define \defi{almost-equidistant} subsets of a metric space, of which almost-equiangular sets are a special case.

Denote the maximum cardinality of a $t$-almost-equiangular set in~$S^{n-1}$ by~$\alpha(n, t)$.  The problem of determining~$\alpha(n, t)$ is called the \defi{$t$-almost-equiangular-set problem}.  For~$t = 0$, this problem  first appears in a paper by Rosenfeld~\cite{Rosenfeld1991AlmostEd}, who attributes the question to Erdős. He showed that~$\alpha(n, 0) = 2n$; a lower bound is given by the union of two disjoint orthogonal bases and an upper bound is given through an interesting argument involving the spectrum of a matrix associated to an almost-equiangular set. Pudlák~\cite{Pudlak2002CyclesMatrices} and Deaett~\cite{Deaett2011TheGraph} reproved this result by slightly simpler methods.

Later, Bezdek and Lángi~\cite{Bezdek1999AlmostSd-1} extended Rosenfeld's spectral bound to~$t \in [-1, \varepsilon]$, where~$\varepsilon > 0$ is a number close to~$0$ that depends on the dimension. In particular, they proved that~$\alpha(n, t) \leq 2(n+1)$ on this interval with equality at~$t = -1/n$. An example of an optimal construction at this inner product is the union of two disjoint regular $n$-simplices. Polyanskii~\cite{Polyanskii2017OnII} mentioned a simple lifting argument to obtain~$\alpha(n,t) \leq 2(n+1)$ for~$t \leq 0$ directly from Rosenfeld's original result.

The goal of the current work is two-fold.  First, to obtain better upper bounds on the number~$\alpha(n, t)$, which is done through semidefinite programming and through closer investigation of the spectral bound of Bezdek and Lángi. Both methods reproduce known bounds, and improve many others.  Second, to list all $t$-almost-equiangular subsets of~$S^{n-1}$ of size~$\alpha(n, t)$ for small~$n$. The spectral bound of Bezdek and Lángi again plays an important role; it is used to derive characterizing properties of those $t$-almost-equidistant sets in~$S^{n-1}$ that are maximum for all~$t \in [-1,0]$.


\subsection{Upper bounds through semidefinite programming}

For~$t \in [-1, 1)$, the \defi{equiangular-lines problem} asks for the maximum number of vectors in~$S^{n-1}$ such that any two distinct vectors have inner product~$\pm t$.  This problem can be rephrased in terms of independent sets of graphs.

Indeed, let~$G = (V, E)$ be a graph.  A subset of~$V$ is \defi{independent} if no two distinct vertices in it are adjacent.  The \defi{independence number} of~$G$, denoted by~$\alpha(G)$, is the maximum cardinality of an independent set of~$G$.  Now consider the graph~$G$ whose vertex set is~$S^{n-1}$ and in which distinct points~$x$, $y \in S^{n-1}$ are adjacent if~$x\cdot y \neq \pm t$.  Then independent sets of~$G$ correspond to sets of equiangular lines and vice versa.  It follows that the maximum number of equiangular lines is equal to the independence number~$\alpha(G)$ of~$G$.

This simple connection allowed the development of many optimization bounds for equiangular lines through extensions of the Lovász theta number, a graph parameter introduced by Lovász~\cite{Lovasz1979OnGraph} that gives an upper bound to the independence number of a finite graph (see~\cite{deLaat2021K-PointLines, deLaat2022TheAngle} and references therein).  Extensions of the Lovász theta number are behind many bounds for geometrical parameters, such as the linear programming bound for the kissing number~\cite{Delsarte1977SphericalDesigns}, the Cohn-Elkies bound for the sphere packing density~\cite{Cohn2003NewI}, and the chromatic number of Euclidean space~\cite{deOliveiraFilho2010FourierRn, Bachoc2009LowerNumbers}.

The $t$-almost-equiangular-set problem can be rephrased in terms of independent sets of hypergraphs.  Let~$H = (V, E)$ be an $r$-uniform hypergraph for some~$r \geq 2$ (that is, the edges are $r$-subsets of~$E$).  A subset of~$V$ is \defi{independent} if it does not contain an edge; the \defi{independence number} of~$H$, denoted by~$\alpha(H)$, is the maximum cardinality of an independent set of~$H$.

Given~$n \geq 2$ and~$t \in [-1, 1)$, let~$H(n, t)$ be the $3$-uniform hypergraph whose vertex set is~$S^{n-1}$ and in which a 3-set~$\{x, y, z\}$ of points is an edge if~$t \notin \{x\cdot y, x\cdot z, y\cdot z\}$.  Then independent sets of~$H$ correspond to $t$-almost-equiangular sets and vice versa.  It follows that~$\alpha(n, t) = \alpha(H(n, t))$.

This connection  again opens the door to the development of optimization upper bounds for~$\alpha(n, t)$.  Castro-Silva, Oliveira, Slot and Vallentin~\cite{Castro-Silva2023AHypergraphs} proposed an extension of the Lovász theta number to finite hypergraphs.  It is based on recursion: the theta number of an $r$-uniform hypergraph is defined in terms of the theta number of its \defi{links}, which are $(r-1)$-uniform hypergraphs.  A further extension to infinite hypergraphs by the same authors~\cite{Castro-Silva2022ASets} has applications in Euclidean Ramsey theory.  The current paper proposes an alternative extension of the theta number to infinite hypergraphs like~$H(n, t)$ based on the Lasserre hierarchy and the $k$-point bound~\cite{deLaat2015AGeometry, deLaat2021K-PointLines}.  This bound is strongly related to the semidefinite programming methods developed in~\cite{Bilyk2024OptimizersSets, Bilyk2023OptimalPotentials}, where the authors use similar techniques to reprove Rosenfeld's original bound, and further apply them to energy minimization questions on hypergraphs.

This allows for the computation of upper bounds for~$\alpha(n, t)$ through the use of sums of squares and semidefinite programming.  Analytic bounds can be obtained by interpolating solutions of the resulting semidefinite programming problems, leading to the following theorem proved in Section~\ref{sec:Bounds}.

\begin{theorem}%
\label{thm:interpolation}
    If~$t \in [-1, 0]$ and~$n \geq 3$, and if~$f(n,t) \coloneqq p^2n(1-t)^2/(2(nt^2+1))$, where
    \[
        p \coloneqq \frac{8n^2t^4(2n-1) - 9n^2t^3(n-1) + (2nt^2 -3t + 4)(7n+1)}{2(1-t)(1+7n - 2n^2t^3(2n-1))},
    \]
    then $\alpha(n,t) \leq \floor{f(n,t)} \leq \floor{(16t-9)^2/(128t^2)}$.
\end{theorem}

The bound~$(16t-9)^2/(128t^2)$ in this theorem is an asymptotic bound; it is the limit of~$f(n,t)$ as~$n$ goes to infinity. That there exists an upper bound that does not depend on the dimension~$n$ is consistent with the existence of the constructions considered in this paper not explicitly depending on the embedding dimension~$n$ if~$t$ is far enough removed from~$-1/n$.


\subsection{Lower bounds through constructions}

If~$S$ is a $t$-almost-equiangular set, then its \defi{distance-$t$ graph}, namely the graph with vertex set~$S$ in which~$x$ and~$y$ are adjacent if~$x\cdot y = t$, is \defi{anti-triangle free}, that is, its complement does not contain triangles.  

Necessary and sufficient conditions for some anti-triangle-free graphs to be the distance graph of an almost-equiangular set are given in Section~\ref{sec:realizability}.  Together with the optimization bound of Theorem~\ref{thm:interpolation} and the results of Section~\ref{sec:properties-and-uniqueness}, this leads to constraints for the existence of $t$-almost-equiangular sets of certain sizes, making it possible to list all optimal such sets for dimensions~$n = 2$ and~$3$.  This search leads to the optimal constructions listed in Section~\ref{sec:low-dim} and summarized in Figure~\ref{fig:dim2and3}. 

\begin{figure}[t]%
    \centerfloat
    \includegraphics[width=1.2\textwidth]{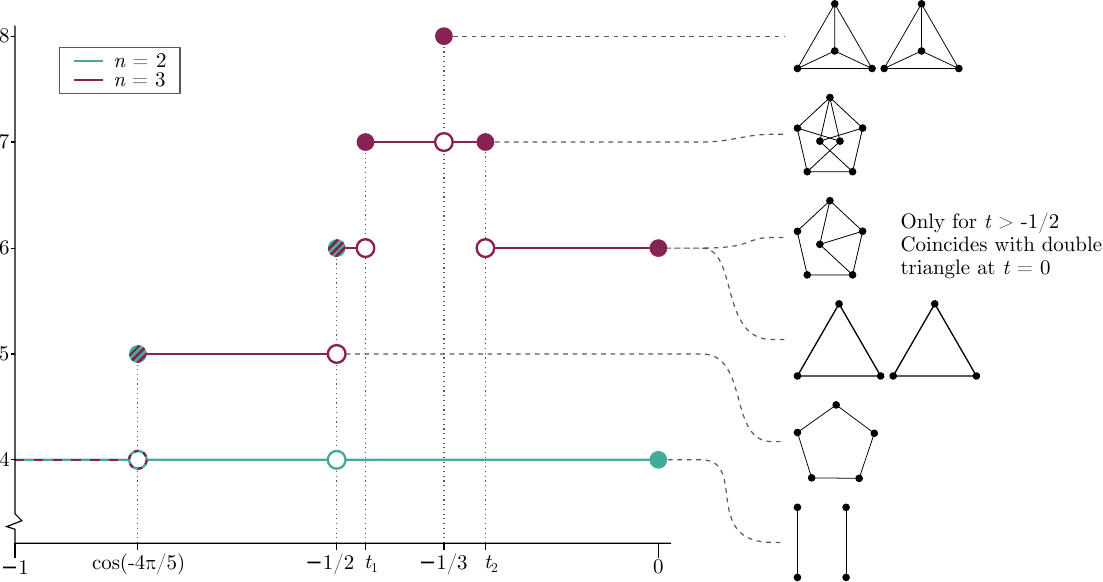}
    \caption{The classification of maximum-cardinality $t$-almost-equiangular sets in~$S^1$ and~$S^2$ for~$t \in [-1, 0]$. The vertical axis is the cardinality of the set, the horizontal axis is the inner product~$t$. Open bullets indicate that a point is excluded from the interval while closed bullets indicate that the point is included. Green and purple stripes indicate that in both dimensions the same maximum cardinality is attained. The numbers~$t_1$ and~$t_2$ are the first two roots of~\eqref{eq:MoserSpindlePolynomial} with~$k=2$.}
    \label{fig:dim2and3}
\end{figure}


\subsection{Maximum obtuse almost-equiangular sets}

Both the semidefinite programming bound of Theorem~\ref{thm:interpolation} and the spectral bound of Rosenfeld~\cite{Rosenfeld1991AlmostEd} and Bezdek and Lángi~\cite{Bezdek1999AlmostSd-1} show that~$\alpha(n,t) \leq 2(n+1)$ for all~$n$ and~$t \leq 0$. In Section~\ref{sec:properties-and-uniqueness} the spectral bound is investigated further to show that equality for nonpositive~$t$ is only attained at~$t = -1/n$. In light of this, call the maximum $(-1/n)$-almost-equidistant sets on~$S^{n-1}$ \defi{maximum obtuse almost-equiangular sets}. Inspection of the matrices that are associated to the maximum obtuse almost-equiangular sets in the proof of the spectral bound reveals several interesting properties of these sets, like the following result.

\begin{theorem}%
\label{cor:optimal-construction-are-2-designs}
        If~$S=\{x_1,\ldots, x_{2(n+1)}\}$ is a maximum obtuse almost-equiangular subset of $S^{n-1}$, then~$S$ is a spherical $2$-design.
\end{theorem}

Deaett proved~\cite{Deaett2011TheGraph} that there is a bijection between the maximum almost-orthogonal sets in~$S^{n-1}$ and certain~$2n \times 2n$ symmetric orthogonal matrices. Any $t$-almost-equidistant set in~$S^{n-1}$ with~$t \leq 0$ can be lifted to an almost-orthogonal set on~$S^n$~\cite{Polyanskii2017OnII}, and so it is expected that there is a version of this bijection for maximum nonpositive almost-equidistant sets as well. The bijection is made precise in the following theorem, which  is Deaett's correspondence with the addition of the eigenvector condition~(1).  Here,~$e$ is the all-ones vector.

\begin{theorem}%
\label{thm:O-matrix}
There is a bijection between maximum obtuse almost-equiangular subsets of $S^{n-1}$ up to orthogonal transformations and $2(n+1) \times 2(n+1)$ symmetric, orthogonal matrices~$O$ such that
\begin{enumerate}
\item $Oe=e$;

\item $O_{ii}=0$ for all~$i$;

\item $O_{ij}O_{jk}O_{ki}=0$ for all $i$, $j$, and~$k$.
\end{enumerate}
\end{theorem}

Call the union of two disjoint regular  $n$-simplices a \defi{double regular $n$-simplex}. It remains an open question whether a maximum obtuse almost-equiangular set is always a double regular $n$-simplex. However, with the help of Theorem~\ref{thm:O-matrix} the question is settled for $2 \leq n \leq  5$.

\begin{theorem}%
\label{thm:double-simplex-uniqueness}
If~$2 \leq n \leq 5$, then any maximum obtuse almost-equiangular set in~$S^{n-1}$ is a double regular $n$-simplex.
\end{theorem}

\section{Preliminaries}%
\label{sec:prelim}

The Euclidean inner product between~$x$ and~$y \in \R^n$ is denoted by~$x\cdot y$.  The \defi{trace inner product} between matrices~$A$, $B \in \R^{n \times n}$ is denoted by~$\langle A, B\rangle = \trace A^\tp B$.  The $(n-1)$-dimensional unit sphere is denoted by~$S^{n-1} = \{\,x \in \R^n : \|x\|=1\,\}$.

Let~$X$ be a topological space.  The set of real-valued continuous functions on~$X$ is denoted by~$C(X)$.  A \defi{kernel} on~$X$ is a function in~$C(X^2)$.  A kernel~$K \in C(X^2)$ is \defi{positive semidefinite} if for every finite set~$U \subseteq X$ the matrix~$\bigl(K(x, y)\bigr)_{x, y \in U}$ is positive semidefinite.

\subsection{Hypergraphs}

A \defi{hypergraph} is a pair~$H = (V, E)$ where~$V$ is a set and~$E$ is a collection of subsets of~$V$.  The set~$V$, also denoted by~$V(H)$, is the \defi{vertex set} of~$H$, and the set~$E$, also denoted by~$E(H)$, is the \defi{edge set} of~$H$.  Elements of~$V$ are called \defi{vertices} and elements of~$E$ are called \defi{edges}.  The hypergraph is called \defi{$r$-uniform} if all edges have cardinality~$r$.  All hypergraphs considered in this paper are $r$-uniform for some~$r$,  so the adjective ``$r$-uniform'' will often be omitted.

Let~$H = (V, E)$ be a hypergraph.  Given~$S \subseteq V$, the subgraph of~$H$ \defi{induced} by~$S$, denoted by~$H[S]$, is the hypergraph with vertex set~$S$ whose edges are all edges of~$H$ contained in~$S$.  For any set~$S \subseteq V$, write~$H - S \coloneqq H[V\setminus S]$.

A subset of~$V$ is \defi{independent} if it contains no edge of~$H$.  The \defi{independence number} of~$H$, denoted by~$\alpha(H)$, is the maximum cardinality of an independent set.  The set of all independent sets of~$H$ of size~$\leq k$ is denoted by~$\inds{k}$.  The set of all independent sets of size~$k$ is denoted by~$\inds{=k}$.  When these notations are used, the hypergraph will be clear from context.

\subsection{Spaces of subsets}%
\label{sec:spaces-of-subsets}

Denote by~$\sub{V}{k}$ the set of all subsets of~$V$ of cardinality at most~$k$, including the empty set. Let
\[
    \set{\,} \colon V^k \to \sub{V}{k} \setminus \{ \emptyset \}
\]
be the map that sends a~$k$-tuple to the set of its coordinates, so~$\set{(s_1, \ldots, s_k)} = \{s_1, \ldots, s_k\}$. This map is surjective. If~$V$ is a topological space, the \defi{standard topology} on~$\sub{V}{k} \setminus \{\emptyset\}$ is the quotient topology under~$\set{\,}$; see Handel~\cite{Handel2000SomeSpaces} for more background on the standard topology.

Let~$H = (V, E)$ be a hypergraph where~$V$ is a topological space.  It will be necessary to work with continuous functions on~$\inds{k} \subseteq \sub{V}{k}$, and for that it is necessary to equip~$\inds{k}$ with a topology.

A natural choice is the relative topology from~$\sub{V}{k}$, but this places an unnecessary restriction on the continuous functions that can be considered.  An alternative is to give~$\inds{=i}$ the relative topology and to equip~$\inds{k}$ with the disjoint union topology of the topological spaces~$\inds{=i}$, that is,~$\inds{k}=\coprod_{i=0}^k \inds{=i}$.  This is equivalent to defining a function on~$\inds{k}$ to be continuous if it is piecewise continuous on each~$\inds{=i}$, that is, the set of continuous functions~$C(\inds{k})$ can be identified with
\[
    \bigoplus_{i=0}^k C(\inds{=i}).
\]

\subsection{Geometry}%
\label{sec:prelim-geometry}

Given~$V \subseteq S^{n-1}$ and~$t \in [-1, 1)$, the \defi{distance-$t$ graph} of~$V$ is the graph whose vertex set is~$V$ and in which~$x$ and~$y$ are adjacent if~$x\cdot y = t$.  A graph~$G = (V, E)$ is \defi{$(n, t)$-realizable} if there is an injection~$f\colon V \to S^{n-1}$ such that~$f(x) \cdot f(y) = t$ for every~$xy \in E$.  If~$xy \notin E$, then there is no constraint on~$f(x) \cdot f(y)$.

An \defi{$(n - 1)$-sphere} is a translated and scaled copy of~$S^{n-1}$. Let~$S$ be an $(n-1)$-sphere~$S$ with radius~$r$, and let~$k \leq n - 1$. A \defi{great $k$-sphere of~$S$} is a $k$-sphere with radius~$r$ contained in~$S$. A great $k$-sphere of~$S^{n-1}$ is then the intersection of~$S^{n-1}$ with a~$(k + 1)$-dimensional linear subspace of~$\R^n$. A great $1$-sphere is called a \defi{great circle}.

An \defi{~$n$-simplex} is the convex hull of~$n + 1$ affinely independent points in Euclidean space.  An $n$-simplex is often identified with its set of~$n+1$ vertices.
A \defi{regular $n$-simplex} with \defi{inner product~$t$} is a regular simplex whose vertices all lie on a unit sphere and have pairwise inner product~$t$. The $t$-distance graph of a regular~$n$-simplex with inner product~$t$ is isomorphic to~$\comp{n+1}$, the complete graph on~$n+1$ vertices.  Conversely, for all~$n \geq k$ and~$d > 0$, the graph~$\comp{k+1}$ is $(n,t)$-realizable, and its realization is a regular $k$-simplex with inner product~$t$.

The \defi{circumsphere} of an~$n$-simplex in~$\R^n$ is the unique sphere that goes through all the vertices of the simplex~\cite[Section 1.4]{Fiedler2011MatricesGeometry}. For~$S \subseteq \R^n$, let~$\Aff{S}$ denote  the affine span of~$S$. In general, if~$S$ is a~$k$-simplex contained in~$\R^{n}$, define its circumsphere as the circumsphere of~$S$ in~$\Aff{S}$.  With this definition, the circumsphere of a $k$-simplex in~$\R^n$ is unique.


\section{Optimization bounds for the independence number of hypergraphs}
\label{sec:optimization-bounds}

The Lovász theta number~\cite{Lovasz1979OnGraph}, defined as the optimal solution of a
semidefinite program, gives an upper bound for the independence number of a
graph.  It can be extended to give an upper bound for the independence number of
hypergraphs as well.  One such extension was proposed by Castro-Silva, Oliveira,
Slot, and Vallentin~\cite{Castro-Silva2023AHypergraphs} by using the theta number recursively.
This section offers two more extensions, based on the Lasserre
hierarchy~\cite{Lasserre2001GlobalMoments, Lasserre2002AnPrograms, Laurent2003AProgramming}, that are better suited for the
almost-equiangular-set problem considered in this paper.

Let~$H = (V, E)$ be a finite $r$-uniform hypergraph and for integer~$k \geq 0$ denote
by~$\inds{k}$ the set of all independent sets of~$H$ of cardinality at most~$k$.
Let~$M_r\colon \R^{\inds{2r}} \to \R^{\inds{r}^2}$ be the operator such that
\[
  (M_r\nu)(S, T) \coloneqq \begin{cases}
    \nu(S \cup T)&\text{if $S \cup T$ is independent;}\\
    0&\text{otherwise}
  \end{cases}
\]
and consider the optimization problem
\begin{equation}%
  \label{opt:lasserre-primal}
  \begin{optprob}
    \lassn(H) \coloneqq \max&\sum_{x \in V} \nu(\{x\})\\
    &\nu(\emptyset) = 1,\\
    &M_r\nu\text{ is positive semidefinite,}\\
    &\nu \in \R_+^{\inds{2r}}.
  \end{optprob}
\end{equation}

The optimal value of this problem is an upper bound for the independence number
of~$H$.  Indeed, let~$I \subseteq V$ be independent and set
\begin{equation}%
  \label{eq:nu-def}
  \nu(S) \coloneqq \begin{cases}
    1&\text{if~$S \subseteq I$;}\\
    0&\text{otherwise}
  \end{cases}
\end{equation}
for all~$S \in \inds{2r}$.  Then~$\nu$ is a feasible solution
of~\eqref{opt:lasserre-primal} with objective value~$|I|$, whence~$\lassn(H)
\geq \alpha(H)$. 

Note that~\eqref{opt:lasserre-primal} looks very much like the~$r$th level of
the Lasserre hierarchy for the independent-set problem of graphs, the difference
being which sets are considered independent.  It is also possible to define
a converging hierarchy of better and better bounds for the independence number
of~$H$ by allowing larger subsets in~\eqref{opt:lasserre-primal}, that is, by
taking~$\inds{2k}$ for~$k > r$; such a hierarchy is not explored in this paper.

Now let~$N_r\colon \R^{\inds{r}} \to \R^{\inds{1}^2\times \inds{r - 2}}$ be the
operator such that
\[
  (N_r\nu)(S, T, Q) \coloneqq \begin{cases}
    \nu(S \cup T \cup Q)&\text{if $S \cup T \cup Q$ is independent;}\\
    0&\text{otherwise.}
  \end{cases}
\]
Here~$S$ and~$T$ are either the empty set or singletons and every singleton set is independent.

A function~$A\colon \inds{1}^2 \times \inds{r-2} \to \R$ is \textit{positive
semidefinite} if for every~$Q \in \inds{r-2}$ the matrix~$(S, T) \mapsto A(S, T,
Q)$ is positive semidefinite.  Consider the problem
\begin{equation}%
  \label{opt:kpb-primal}
  \begin{optprob}
    \kpbn(H) \coloneqq \max&\sum_{x \in V} \nu(\{x\}),\\
    &\nu(\emptyset) = 1,\\
    &N_r\nu\text{ is positive semidefinite,}\\
    &\nu \in \R_+^{\inds{r}}.
  \end{optprob}
\end{equation}
Again,~$\Delta(H) \geq \alpha(H)$.  Indeed, if~$I$ is an independent set of~$H$,
then~$\nu \in \R^{\inds{r}}$ given by~\eqref{eq:nu-def} is a feasible solution
of~\eqref{opt:kpb-primal} with objective value~$|I|$.

Actually, if~$\nu$ is a feasible solution of~\eqref{opt:lasserre-primal}, then
its restriction to~$\inds{r}$ is a feasible solution of~\eqref{opt:kpb-primal},
and so~$\kpbn(H) \geq \lassn(H)$.  Problem~\eqref{opt:kpb-primal} is related to
the $r$-point bound for the independence number of a graph~\cite{deLaat2021K-PointLines}.
As before, it is possible to define a converging hierarchy of problems by
allowing larger subsets, that is, by taking~$N_k$ and~$\inds{k}$ for~$k > r$
in~\eqref{opt:kpb-primal}.

The goal now is to extend~\eqref{opt:kpb-primal} to infinite hypergraphs on the
sphere. To this end, the dual problem is more suitable.
Problem~\eqref{opt:kpb-primal} is a conic programming problem (see the book by
Barvinok~\cite{Barvinok2002AConvexity} for background) in which the variable~$\nu$ is required
to belong to the cone
\[
  C \coloneqq \{\,\nu \in \R_+^{\inds{r}} : N_r\nu\text{ is positive semidefinite}\,\}.
\]
The dual cone is
\[
  C^* = \{\, N_r^* A : A \in \R^{\inds{1}^2 \times \inds{r-2}}\text{ is positive
  semidefinite}\,\} + \R_+^{\inds{r}},
\]
where~$N_r^*\colon \R^{\inds{1}^2 \times \inds{r-2}} \to \R^{\inds{r}}$ is the
adjoint of~$N_r$.  Explicitly, the adjoint is given by
\begin{equation}%
  \label{eq:N-adjoint}
  (N_r^*A)(I) = \sum_{Q \in \inds{r - 2}} \sum_{\substack{S, T \in \inds{1}\\Q
  \cup S \cup T = I}} A(S, T, Q)
\end{equation}
for~$I \in \inds{r}$.

The dual of~\eqref{opt:kpb-primal} is then
\begin{equation}%
  \label{opt:kpb-dual}
  \begin{optprob}
    \min&(N_r^* A)(\emptyset)\\
    &(N_r^*A)(\{x\}) \leq -1&\text{for all~$x \in V$,}\\
    &(N_r^*A)(S) \leq 0&\text{for all~$S \in \inds{r}$ with~$|S| \geq 2$,}\\
    &\onerow{A \in \R^{\inds{1}^2 \times \inds{r-2}}\text{ is positive
    semidefinite.}}
  \end{optprob}
\end{equation}
Any feasible solution of the dual problem provides an upper bound
for~$\alpha(H)$, as follows from the weak duality relation.


\subsection{Hypergraphs on the sphere}

The next step is to extend~\eqref{opt:kpb-dual} to hypergraphs on the sphere,
here focusing on hypergraphs representing the almost-equiangular-set problem ---
though the theory extends to $r$-uniform hypergraphs, under some assumptions.

For integer~$n \geq 2$ and~$t \in [-1, 1)$, let~$H = H(n, t)$ be the $3$-uniform hypergraph whose vertex set is~$S^{n-1}$
and in which three distinct points~$x$, $y$, and~$z$ form an edge if $t
\notin \{x\cdot y, x\cdot z, y\cdot z\}$.  Then the $t$-almost-equiangular sets are exactly the independent sets of~$H$, and so the goal is to compute
the independence number~$\alpha(H) = \alpha(n, t)$ of~$H$. 

To define an analogue of~\eqref{opt:kpb-dual} for the infinite hypergraph~$H$,
define the operator
\[
  B_3\colon C(\inds{1}^3) \to \bigoplus_{k=0}^3 C(\inds{=3})
\]
simply by copying~\eqref{eq:N-adjoint}, that is,
\[
  (B_3 A)(I) \coloneqq \sum_{\substack{S, T, Q \in \inds{1}\\S \cup T \cup Q = I}} A(S,
  T, Q)
\]
for~$I \in \inds{3}$.  This is the analogue of~$N_3^*$.  (See Section~\ref{sec:spaces-of-subsets} for background on the topology on the spaces of independent sets.)

Under the topology defined in Section~\ref{sec:spaces-of-subsets}, the operator~$B_3$ is well defined, that is, that~$B_3 A$ is continuous whenever~$A$ is
continuous.  A proof of this fact is simple but technical~\cite{Bekker2023OnGraphs}.

Next, say~$A \in C(\inds{1}^3)$ is \textit{positive semidefinite} if
the kernel~$(S, T) \mapsto A(S, T, Q)$ is positive semidefinite for every~$Q \in
\inds{1}$.  The extension of~\eqref{opt:kpb-dual} then is
\begin{equation}%
  \label{opt:sphere-kpb-dual}
  \begin{optprob}
    \min&(B_3 A)(\emptyset)\\
    &(B_3 A)(\{x\}) \leq -1&\text{for all~$x \in S^{n-1}$,}\\
    &(B_3 A)(S) \leq 0&\text{for all~$S \in \inds{3}$ with~$|S| \geq 2$,}\\
    &\onerow{A \in C(\inds{1}^3)\text{ is positive semidefinite.}}
  \end{optprob}
\end{equation}

\begin{theorem}
  If~$A$ is a feasible solution of~\eqref{opt:sphere-kpb-dual},
  then~$\alpha(H(n, t)) \leq (B_3 A)(\emptyset)$.
\end{theorem}

\begin{proof}
Let~$I \subseteq S^{n-1}$ be an independent set of~$H(n, t)$.  On the one
hand,
\[
  \begin{split}
    \sum_{\substack{J \subseteq I\\|J| \leq 3}} (B_3 A)(J) &=
    \sum_{\substack{J \subseteq I\\|J| \leq 3}} \sum_{\substack{S, T, Q \in
    \inds{1}\\S \cup T \cup Q = J}} A(S, T, Q)\\
    &=\sum_{\substack{S, T, Q \subseteq I\\|S|, |T|, |Q| \leq 1}} A(S, T, Q)\\
    &\geq 0,
  \end{split}
\]
where the last inequality follows from~$A$ being positive semidefinite.

On the other hand,
\[
  \sum_{\substack{J \subseteq I\\|J| \leq 3}} (B_3 A)(J) =
  (B_3 A)(\emptyset) + \sum_{x \in I} (B_3 A)(\{x\}) + \sum_{\substack{J
  \subseteq I\\|J| \geq 2}} (B_3)(J)
  \leq (B_3 A)(\emptyset) - |I|,
\]
whence~$|I| \leq (B_3 A)(\emptyset)$.
\end{proof}


\subsection{Solving the optimization problem}

It simplifies notation to identify~$\inds{1}$ with~$\{\emptyset\} \cup S^{n-1}$, so below~$x \in \inds{1}$ is either~$\emptyset$ or an element
of~$S^{n-1}$.

The orthogonal group~$\ortho(n)$ acts on~$S^{n-1}$ by rotation.  Extend this
action to~$\inds{1}$ by acting trivially on~$\emptyset$.  A function~$A \in
\mc{C}(\inds{1}^3)$ is \defi{$\ortho(n)$-invariant} if
\[
  A(Tx, Ty, Tz) = A(x, y, z)
\]
for all~$x$, $y$, $z \in \inds{1}$ and~$T \in \ortho(n)$.  Invariance for other
functions and groups is similarly defined.

Any feasible solution of~\eqref{opt:sphere-kpb-dual}, and in particular any $\ortho(n)$-invariant feasible solution, gives an upper bound
for~$\alpha(H)$, where~$H = H(n, t)$.  Moreover, nothing is lost by
restricting~\eqref{opt:sphere-kpb-dual} to invariant solutions.  Indeed, every
rotation in~$\ortho(n)$ is an automorphism of~$H$.  It follows that, if~$A$ is a
feasible solution of~\eqref{opt:sphere-kpb-dual}, then
\[
  \overline{A}(x, y, z) \coloneqq \int_{\ortho(n)} A(Tx, Ty, Tz)\, d\mu(T),
\]
where~$\mu$ is the Haar probability measure on~$\ortho(n)$, is an $\ortho(n)$-invariant
feasible solution providing the same bound as~$A$.

Invariant positive-semidefinite functions in~$C(\inds{1}^3)$ can be parameterized
using spherical harmonics.  To see how, consider a positive-semidefinite $\ortho(n)$-invariant 
function~$A\colon \inds{1}^3 \to \R$.  For~$x$, $y \in \inds{1}$, the kernel
$K_\emptyset\colon \inds{1}^2 \to \R$ defined by
\[
  K_\emptyset(x, y) \coloneqq A(x, y, \emptyset)
\]
is positive semidefinite and $\ortho(n)$-invariant.

Now fix~$e \in S^{n-1}$ and let~$z \in S^{n-1}$.  There is~$T \in \ortho(n)$ such
that~$Tz = e$, so~$A(x, y, z) = A(Tx, Ty, e)$.  Let~$K_e\colon \inds{1}^2 \to
\R$ be the kernel such that
\[
  K_e(x, y) \coloneqq A(x, y, e).
\]
This kernel is positive semidefinite and invariant under the
\defi{stabilizer subgroup of~$e$}, namely the subgroup~$\stab{e}$ of~$\ortho(n)$
that fixes~$e$.

It follows that an $\ortho(n)$-invariant positive-semidefinite function~$A \in
C(\inds{1}^3)$ can be represented by two positive-semidefinite kernels
in~$C(\inds{1}^2)$, namely~$K_\emptyset$ and~$K_e$, the kernel~$K_\emptyset$
being $\ortho(n)$-invariant and the kernel~$K_e$ being $\stab{e}$-invariant.  The
correspondence is simply
\begin{align*}
  A(x, y, \emptyset) &= K_\emptyset(x, y)\quad\text{and}\\
  A(x, y, z) &= K_e(Tx, Ty),
\end{align*}
where~$T$ is any element of~$\ortho(n)$ such that~$Tz = e$.  It follows from the
invariance of~$K_e$ that~$A$ is well defined, since if~$T_1 z = T_2 z = e$,
then~$T_2 T_1^{-1} \in \stab{e}$ and $K_e(T_1 x, T_1 y) = K_e(T_2 x, T_2 y)$.

Schoenberg's theorem~\cite{Schoenberg1942PositiveSpheres} characterizes $\ortho(n)$-invariant
positive-semidefinite kernels on~$S^{n-1}$ in terms of 
Gegenbauer polynomials.  Similarly, a theorem of Bachoc and
Vallentin~\cite{Bachoc2008NewProgramming} characterizes $\stab{e}$-invariant
positive-semidefinite kernels on~$S^{n-1}$ using multivariate Gegenbauer
polynomials.  Both characterizations can be easily adapted to kernels
on~$\inds{1}$.  For this the following lemma is useful.

\begin{lemma}%
  \label{lem:psd-exp}
  Let~$V$ be a topological space,~$f_1$, \dots,~$f_N\colon V \to \R$ be
  continuous functions, and for~$x$, $y \in V$ consider the~$N \times N$ matrix
  such that
  \[
    Z(x, y)_{ij} \coloneqq f_i(x) f_j(y).
  \]
  If~$A \in \R^{N \times N}$ is positive semidefinite, then the kernel~$K\colon
  V^2 \to \R$ such that
  \[
    K(x, y) \coloneqq \langle A, Z(x, y)\rangle
  \]
  is positive semidefinite.
\end{lemma}

\begin{proof}
Let~$x_1$, \dots,~$x_k \in V$ and take~$u \in \R^k$.  Since~$A$ is positive semidefinite, the matrix~$A \otimes u u^\tp$ is also positive semidefinite; its rows and columns are indexed by~$I = \{1, \ldots, N\} \times \{1, \ldots, k\}$.  Setting~$g(i, k) \coloneqq f_i(x_k)$ it follows that
\[
    \begin{split}
        \sum_{k,l=1}^k K(x_k, x_l) u_k u_l &= \sum_{k, l=1}^k u_k u_l \sum_{i,j=1}^N A_{ij} f_i(x_k) f_j(x_l)\\
        &\sum_{(i,k), (j, l) \in I} (A \otimes uu^\tp)_{(i,k), (j,l)} g(i, k) g(j, l)\\
        &\geq 0,
    \end{split}
\]
as wanted.
\end{proof}

Start with~$K_\emptyset$.  Let~$P_k^n$ denote the Jacobi polynomial of
degree~$k$ with parameters~$\alpha = \beta = (n-3)/2$ normalized so~$P_k^n(1) =
1$.  For~$k \geq 1$, let~$Z^\emptyset_k\colon \inds{1}^2 \to \R$ be such that
\[
  Z^\emptyset_k(x, y) \coloneqq \begin{cases}
    P_k^n(x\cdot y)&\text{if~$x$, $y \in S^{n-1}$;}\\
    0&\text{otherwise.}
  \end{cases}
\]
Let~$Z^\emptyset_0\colon \inds{1}^2 \to \R^{2\times 2}$ be such that, for~$x$,
$y \in S^{n-1}$,
\begin{align*}
    Z^\emptyset_0(\emptyset, \emptyset)& \coloneqq \smallpmatrix{1&0\\0&0},&
  Z^\emptyset_0(x, \emptyset)& \coloneqq \smallpmatrix{0&0\\1&0},\\
  Z^\emptyset_0(\emptyset, x)& \coloneqq \smallpmatrix{0&1\\0&0},&
  Z^\emptyset_0(x, y)& \coloneqq \smallpmatrix{0&0\\0&1}.
\end{align*}

It follows from the addition formula for Gegenbauer polynomials~\cite[\S9.6]{Andrews1999SpecialFunctions} that
for every~$k > 0$ the kernel~$(x, y) \mapsto Z^\emptyset_k(x, y)$ is positive
semidefinite.  From Lemma~\ref{lem:psd-exp} it follows that if~$A \in \R^{2\times 2}$ is positive semidefinite, then the kernel $(x, y) \mapsto \langle
A, Z^\emptyset_0(x, y)\rangle$ is positive semidefinite.  So, for every~$d \geq
0$, any kernel of the form
\begin{equation}%
  \label{eq:kempty-kernel}
  (x, y) \mapsto \langle A_0^\emptyset, Z^\emptyset_0(x, y)\rangle + \sum_{k=1}^d a_k
  Z^\emptyset_k(x, y)
\end{equation}
for positive-semidefinite~$A_0^\emptyset \in \R^{2 \times 2}$ and nonnegative numbers~$a_k$
is $\ortho(n)$-invariant and positive semidefinite.
Schoenberg~\cite{Schoenberg1942PositiveSpheres} showed the same for kernels on~$S^{n-1}$, the
only difference is that~$Z^\emptyset_0$ is then a single number. As in the case
of Schoenberg's theorem, it is possible to show that any $\ortho(n)$-invariant
and positive-semidefinite kernel in~$C(\inds{1}^2)$ can be written in the form
above with uniform convergence if one allows for an infinite series instead of a
finite sum.  Here, this full characterization is not needed.

Next consider~$K_e$.  Bachoc and Vallentin~\cite{Bachoc2008NewProgramming} define the
multivariate Gegenbauer polynomial, for~$k \geq 0$, as
\[
  Q_k^n(u, v, t) \coloneqq (1 - u^2)^{k/2} (1 - v^2)^{k/2} P_k^{n-1}\biggl(\frac{t - uv}{(1
    - u^2)^{1/2} (1 - v^2)^{1/2}}\biggr);
\]
this is a polynomial on~$u$, $v$, and~$t$ of degree~$2k$.

For~$k > 0$ and~$x$, $y \in S^{n-1}$, let~$Z^e_k(x, y)$ be the infinite matrix
indexed by integers~$i$, $j \geq 0$ such that
\[
  Z^e_k(x, y)_{ij} \coloneqq (e\cdot x)^i (e\cdot y)^j Q_k^n(e\cdot x, e\cdot y, x\cdot
  y).
\]
Note that this is a polynomial on~$e\cdot x$, $e\cdot y$, and~$x\cdot y$ of
degree~$i + j + 2k$.  If~$x = \emptyset$ or~$y = \emptyset$, set~$Z^e_k(x,
y)_{ij} \coloneqq 0$.

For integer~$i \geq 0$, let~$f_i\colon \inds{1} \to \R$ be such that
\[
  f_i(x) \coloneqq \begin{cases}
    0&\text{if~$x = \emptyset$;}\\
    (e\cdot x)^i&\text{otherwise.}
  \end{cases}
\]
Let~$f_\emptyset\colon \inds{1} \to \R$ be such that~$f_\emptyset(\emptyset) \coloneqq 1$
and~$f_\emptyset(x) \coloneqq 0$ if~$x \in S^{n-1}$. Define the infinite
matrix~$Z^e_0(x, y)$, indexed by~$U \coloneqq \{\emptyset\} \cup \{\, i \in \Z : i \geq
0\,\}$, by setting
\[
  Z^e_0(x, y)_{\alpha\beta} \coloneqq f_\alpha(x) f_\beta(y)
\]
for~$\alpha$, $\beta \in U$.

Let~$A$ be a positive-semidefinite matrix indexed by a finite set of nonnegative
integers.  For~$k > 0$, Bachoc and Vallentin~\cite{Bachoc2008NewProgramming} showed that the
kernel
\begin{equation}%
  \label{eq:bachoc-v-kernel}
  (x, y) \mapsto \langle A, Z^e_k(x, y)\rangle
\end{equation}
on~$S^{n-1}$ is positive semidefinite.  In the trace inner product in~\eqref{eq:bachoc-v-kernel}, the
matrix~$Z_k^e$ is truncated, that is, only the finite submatrix
corresponding to the rows and columns of~$A$ is considered.  From this it
immediately follows that the kernel~\eqref{eq:bachoc-v-kernel} is positive
semidefinite as a kernel over~$\inds{1}$ as well.

As for~$k = 0$, if~$A$ is a positive-semidefinite matrix indexed by a finite
subset of the index set~$U$, then the kernel $(x, y) \mapsto \langle A, Z^e_k(x,
y)\rangle$ is positive semidefinite, as follows directly from
Lemma~\ref{lem:psd-exp}.  So, if~$A^e_0$, \dots,~$A^e_d$ are positive-semidefinite
matrices, with~$A^e_0$ indexed by a subset of~$U$ and~$A^e_k$ indexed by a subset of
the nonnegative integers for~$k > 0$, then
\begin{equation}%
  \label{eq:ke-kernel}
  K_e(x, y) \coloneqq \sum_{k=0}^d \langle A^e_k, Z^e_k(x, y)\rangle
\end{equation}
is positive semidefinite and, by construction, $\stab{e}$-invariant.  Every
$\stab{e}$-invariant positive-semidefinite continuous kernel~$K_e$ can be
uniformly approximated by kernels with the above expression, see for example the appendix of~\cite{deLaat2019MomentProblems}.

With this, it is possible to express the function~$A \in C(\inds{1}^3)$
of~\eqref{opt:sphere-kpb-dual} in terms of polynomials.  Here,~$d$
in~\eqref{eq:kempty-kernel} and~\eqref{eq:ke-kernel} is fixed and the
matrices~$A_k$ in~\eqref{eq:ke-kernel} are truncated appropriately to bound the
total degree of the polynomials used.  The constraints
of~\eqref{opt:sphere-kpb-dual} are modeled as polynomial constraints using sums
of squares.  In this way,~\eqref{opt:sphere-kpb-dual} can be solved numerically
with the computer, and solutions can even be found analytically.  Both approaches are discussed in Section~\ref{sec:Bounds}.


\section{Upper bounds from the 3-point bound}%
\label{sec:Bounds}

As shown in Section~\ref{sec:optimization-bounds}, the bound~\eqref{opt:sphere-kpb-dual} can be expressed in terms of a polynomial optimization problem once~$d$ is fixed in~\eqref{eq:kempty-kernel} and~\eqref{eq:ke-kernel} and the~$Z_k^\emptyset$ and~$Z_k^e$ matrices are truncated to finite matrices.

\begin{figure}[t]%
    \centerfloat
    \includegraphics{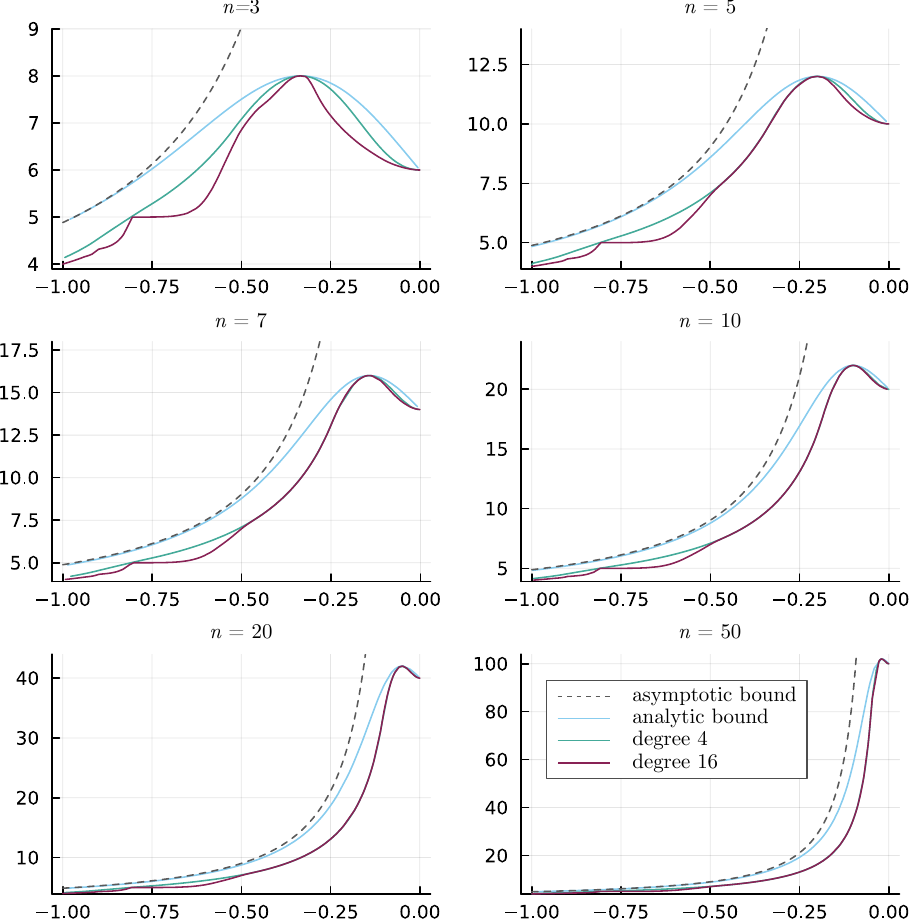}
    \caption{The numeric solutions of degree~$4$ and~$16$, and the analytic (Theorem~\ref{thm:interpolation}) and asymptotic solutions to the bound~\eqref{opt:sphere-kpb-dual}. The asymptotic bound is the limit as $n \to \infty$ of the analytic bound.}
    \label{fig:SolutionToBound}
\end{figure}

So implemented, the 3-point bound~\eqref{opt:sphere-kpb-dual} gives particularly good results for~$t \leq 0$.  Figure~\ref{fig:SolutionToBound} shows a plot of this bound as a function of~$t \in [-1, 0]$; the bound was computed by a Julia program using the package \texttt{ClusteredLowRankSolver.jl}~\cite{Leijenhorst2024SolvingOptimization}.  These are numerical results of very high precision that can be turned into rigorous results with some extra effort.  The Julia package \texttt{AlmostEquiangular.jl}, contained in the arXiv supplement to this paper, includes a function to compute the 3-point bound.

Using \texttt{ClusteredLowRankSolver.jl}~\cite{Leijenhorst2024SolvingOptimization} and its rounding routine~\cite{Cohn2024OptimalityBounds}, it is possible to obtain a rational analytic solution for fixed dimension~$n \geq 3$ and for inner products~$0$ and~$-1 / n$. At these points the bound is exactly equal to the maximum size of an almost-equiangular set. These solutions can then be interpolated to obtain a rational function in~$n$ and~$t$ that gives an upper bound for the size of a $t$-almost-equiangular set in~$S^{n-1}$ for~$t \in [-1, 0]$, leading to Theorem~\ref{thm:interpolation}.

A union of two disjoint regular $(n-1)$-simplices in~$S^{n-1}$ gives a $0$-almost-equiangular set with~$2n$ points; Rosenfeld~\cite{Rosenfeld1991AlmostEd} showed that this construction is optimal.  A union of two disjoint regular $n$-simplices in~$S^{n-1}$ gives a $(-1/n)$-almost-equiangular set with~$2(n+1)$ points; Bezdek and Lángi~\cite{Bezdek1999AlmostSd-1} showed that this construction is optimal.  The bound of Theorem~\ref{thm:interpolation} is sharp in both cases, providing a new proof of the optimality of these constructions.

\begin{proof}[Proof of Theorem~\ref{thm:interpolation}]
The proof of the theorem is by exhibiting a solution to the 3-point bound that has the objective value in the statement. To keep the solution as simple as possible, use a degree-$0$ kernel~$K_{\emptyset}$ and a degree-$4$ kernel~$K_e$. So the set of positive-semidefinite variables is~$A_0^{\emptyset}$ and~$A_k^e$ with~$0 \leq k \leq 2$.

Let
\[
    p \coloneqq \frac{8n^2t^4(2n-1) - 9n^2t^3(n-1) + (2nt^2 -3t + 4)(7n+1)}{2(1-t)(1+7n - 2n^2t^3(2n-1))}
\]
and
\begin{align*}
    &A_0^{\emptyset} \coloneqq 
    \begin{pmatrix}
        \frac{n(1 - t)^2}{2(nt^2 + 1)} p^2  & -\frac{1}{2}p \\ \mathbf{*} & \frac{nt^2 + 1}{2n(1 - t)^2}
    \end{pmatrix},\\
    &A_0^e \coloneqq
    \begin{pmatrix}
        p & -\frac{1}{4 n (1 - t)^2} -\frac{t^2}{(1 - t)^2} & \frac{3 t}{2 (1 - t)^2} & -\frac{3}{4 (1 - t)^2}\\
        \mathbf{*} & \frac{1}{4 n (n - 1) (1 - t)^3} - \frac{t^3}{2 (1 - t)^3} & \frac{3 t^2}{4 (1 - t)^3} & -\frac{n + 1}{8 n (n - 1) (1 - t)^3}\\
        \mathbf{*} & \mathbf{*} & -\frac{3 t}{2 (1 - t)^3} & 0\\
        \mathbf{*} & \mathbf{*} & \mathbf{*} & \frac{2 n - 1}{4 (n - 1) (1 - t)^3}
    \end{pmatrix},\\
    &A_1^e \coloneqq
    \begin{pmatrix}
        0 & 0\\ \mathbf{*} & \frac{n + 1}{2 n (1 - t)^3}
    \end{pmatrix},\\
    &A_2^e \coloneqq \begin{pmatrix}\frac{n-2}{4 n (n - 1) (1 - t)^3}\end{pmatrix}.
\end{align*}
The~$\mathbf{*}$s indicate that the entries are determined by the symmetry of the matrices.

All matrices above, except for~$A_0^e$, can be checked by hand to be positive semidefinite in the domain given by~$n \geq 3$ and~$t \in [-1, 0]$.  To check that~$A_0^e$ is positive semidefinite in the required domain, first decompose it as~$A_0^e = L D L^\tp$, where~$L$ and~$D$ are matrices of rational functions on~$n$ and~$t$ and~$D$ is diagonal, and then check that the diagonal entries of~$D$ are nonnegative in the domain.

These diagonal entries are rational functions, which can be rigorously checked to be nonnegative by a sum-of-squares approach.  The arXiv supplement to this paper contains the Julia package \texttt{AlmostEquiangular.jl}, which provides sum-of-squares certificates for the nonnegativity of the diagonal entries of~$D$.  The same package also provides a sum-of-squares certificate for the inequality~$f(n, t) \leq (16t-9)^2 / (128t^2)$. 

The Julia package in the supplement also checks that, for the corresponding function~$A \in C(\inds{1}^3)$,
\begin{align*}
    &B_3A(\{x\}) = -1,\\
    &B_3A(\{x, y\}) = 0 &&\text{ for all~$x \not= y$, and}\\
    &B_3A(\{x, y, z\}) = \frac{3(t - x \cdot z)(t - y \cdot z)(t - x \cdot y)}{(t - 1)^3} &&\text{ for all~$x$,~$y$, and~$z$ distinct.}   
\end{align*}
In particular, if~$t \in \{x \cdot z, y \cdot z, x \cdot y\}$, then~$B_3A(\{x, y, z\}) = 0$.
\end{proof}

The solution constructed in the proof above can in principle be improved; the issue is to get a good compromise between simplicity and quality.  For instance, by forcing some of the matrix entries to be zero as done above, it becomes possible to find a simple rational expression as given in the theorem.

\section{Realizability of anti-triangle-free graphs}
\label{sec:realizability}

A graph is \defi{anti-triangle free} if its complement is triangle free. This is equivalent to saying that every triple of vertices contains an edge.  The distance graphs of almost-equiangular sets are anti-triangle free and, conversely, realizable anti-triangle-free graphs give almost-equiangular sets.  Hence, to construct good almost-equiangular sets, one has to show that given anti-triangle-free graphs are realizable.

Recall the definition of realizability from Section~\ref{sec:prelim-geometry}.  The goal of this section is to determine whether certain anti-triangle-free graphs are~$(n, t)$-realizable. A construction of interest is the $(k,l)$-spindle, denoted by~$\spindle{k}{l}$ with~$k,l \geq 1$, defined later in this section, of which the Moser spindle is a special case.  In order to bound the inner products at which~$\spindle{k}{l}$ is realizable, and to offer some tools for other calculations, it is useful to derive realizability of some commonly appearing subgraphs of the spindle, namely the simplex and the rhombus.

\subsection{The simplex}%
\label{sec:simplex}
A nice reference for simplex geometry is  Fiedler~\cite{Fiedler2011MatricesGeometry}; see in particular Theorem~4.5.1 of this book for the following facts. The inner products of distinct vertices of a regular $n$-simplex inscribed in~$S^{n - 1}$ is~$-1 / n$. So~$\comp{n+1}$ is~$(n,t)$-realizable if and only if~$t = - 1/ n$.

If~$k < n$, then~$\comp{k + 1}$ is $(n, t)$-realizable if and only if~$t \geq - 1 / k$. Indeed, the circumradius of a regular $k$-simplex with inner product~$t$ is
\[
    r_k(t) \coloneqq \sqrt{\frac{(1 - t) k}{k + 1}}
\]
and the circumsphere of a $k$-simplex is a $(k - 1)$-sphere. For~$k < n$, the sphere~$S^{n-1}$ contains a $(k - 1)$-sphere of every radius less than or equal to~1, so~$K_{k+1}$ is $(n, t)$-realizable if and only if~$r_k(t) \leq 1$. This happens if and only if~$t \geq -1 / k$.

The~$k+1$ vertices of a regular $k$-simplex on~$S^{n-1}$ are by definition affinely independent, and so a regular~$k$-simplex contains at least~$k$ linearly independent points. If~$t = -1/k$, then~$r_k(t) = 1$, and the circumsphere is a great sphere, which lies on a linear subspace of dimension~$k$. However, if~$k < n$ and~$t > -1/k$, then~$r_k(t) < 1$, and so the linear span of the~$k$-simplex has dimension~$k+1$. In this case, the vertices of the $k$-simplex are linearly independent.  

\subsection{The rhombus}

A useful subgraph of a spindle is the union of two complete graphs on~$k+1$ vertices that have exactly~$k$ vertices in common.  This is the distance graph of a pair of regular $k$-simplices that share exactly one facet.  Alternatively, it is the complete graph~$\comp{k+2}$ with one edge removed.  Call this graph a \defi{$k$-rhombus}. By the previous paragraph, necessary conditions for realizability are~$k \leq n$,~$t = -1 / k$ if~$k = n$, and~$t \geq - 1 / k$ otherwise.

In what follows, let~$R$ be a~$k$-rhombus that is the union of two instances of~$\comp{k+1}$, denoted by~$\Sigma_1$ and~$\Sigma_2$, let~$e$ be the unique vertex of~$\Sigma_1 - V(\Sigma_2)$, let~$p$ be the unique vertex in~$\Sigma_2 - V(\Sigma_1)$, and let~$B = V(\Sigma_1) \cap V(\Sigma_2)$. Refer to~$R[B]$ as the \defi{base} of the rhombus. It is an instance of~$\comp{k}$. Up to orthogonal transformations, an~$(n, t)$-realization of~$R[B]$ is uniquely determined, so assume its vectors are known and denote the realization by~$B$ as well. The following lemma is comparable to~\cite[Lemma 7]{Balko2020Almost-EquidistantSets}.

\begin{lemma}%
\label{lem:kRhombusRealizable}
    With~$e$, $p$ as above, the~$k$-rhombus is~$(n,t)$-realizable if and only if~$k \leq n-1$ and~$t > -1 / k$. If these conditions hold, then~$e$ and~$p$ lie on an~$(n - k - 1)$-sphere of radius~$\sqrt{1 - 2kt^2/((k - 1)t + 1)}$. In particular, let~$k \leq n-1$,~$t > -1 / k$, and 
    \[
        \tau \coloneqq \frac{2 k t^2}{(k - 1) t + 1} - 1.
    \]
    If~$k < n - 1$, then~$e \cdot p \geq \tau$, and if~$k = n - 1$, then~$e \cdot p = \tau$.
    
    Conversely, if~$k < n-1$ and~$t' \in [\tau, 1)$, then there exists an~$(n, t)$-realization of the~$k$-rhombus in which~$e \cdot p = t'$.  If~$k = n - 1$, then the points~$e$ and~$p$ are uniquely determined.
\end{lemma}

\begin{proof}
The $k$-rhombus with base~$B$ has a subgraph isomorphic to~$\comp{k+1}$, and so necessary conditions for realizability are~$t \geq - 1/k$ and~$k \leq n$.  Assume that these hold.  If~$k = n$, the vectors in~$B$ already determine a full rank system, so then~$p$ will coincide with~$e$. Consequently, another necessary condition is~$k \leq n-1$.

Since~$t \geq -1 / k$, the~$k$-rhombus is realizable if and only if the affine space
\[
    A \coloneqq \{\, x \in \R^n : x \cdot b = t \text{ for all } b \in B\, \}
\]
intersects~$S^{n-1}$ in more than one point, that is, if and only if~$\inf{\{\|a\| : a \in A\}} < 1$; this infimum is attained in~$A$.

Let~$U$ be the linear span of~$B$ and let~$W$ be its orthogonal complement. Then the shortest vector~$a_0$ in~$A$ is in~$U$. Indeed, if~$a_0 = \sum_{b \in B} \lambda_b b + w \in A$ with~$w \in W$, then by orthogonality~$\|a_0\|^2 = \|\sum_{b \in B} \lambda_b b\|^2 + \|w\|^2$. Translating by a vector orthogonal to~$U$ does not change the inner product with any of the elements in~$B$. So, if~$\|a_0\|$ is minimal, then~$w = 0$.

All that is left is to calculate the coefficients~$\lambda_b$. Since~$A$ is convex,~$a_0$ is the unique shortest vector. Because~$b \cdot b' = t$ for all distinct~$b, b' \in B$, applying a permutation to the coefficients gives another vector in~$A$ with the same norm. By uniqueness, this forces all~$\lambda_b$ to have the same value~$\lambda$. For every~$b \in B$,
\[
    t = a_0 \cdot b = \lambda \sum_{b' \in B} b' \cdot b = \lambda((k - 1) t + 1),
\]
so~$\lambda = t / ((k - 1)t + 1)$ and
\[
    \|a_0\|^2 = \frac{kt^2}{(k - 1)t + 1}.
\]
Since~$t \geq -1 / k$, it follows that~$\|a_0\| < 1$ if and only if~$kt^2 - (k - 1)t - 1 < 0$. As a polynomial in~$t$ it has roots~$1$ and~$-1 / k$, so the~$k$-rhombus is realizable if and only if~$-1 / k < t < 1$.

The intersection of~$A$ with~$S^{n-1}$ gives an~$(n - k - 1)$-sphere~$S$ whose radius~$r$ is~$\sqrt{1 - \|a_0\|^2} = \sqrt{1 - kt^2 / ((k - 1)t + 1)}$.  Any two distinct points on~$S$ are valid realizations of~$p$ and~$e$.  If~$\tau$ is the minimum possible inner product between points on~$S$, then~$2r = \sqrt{2(1 - \tau)}$, so
\[
    \tau = 1 - 2r^2 = \frac{2kt^2}{(k - 1)t + 1} - 1.\qedhere
\]
\end{proof}

In Lemma~\ref{lem:kRhombusRealizable}, the inner product~$t$ does not depend on the embedding dimension, something that happens often for these types of constructions.

\subsection{The spindle}%
\label{sec:SpindleProof}

The \defi{$(k,l)$-spindle} can be described as follows: let~$R_1$ be a $k$-rhombus; say~$e$ and~$p_1$ are the vertices of its unique nonedge. Attach at~$e$ an $l$-rhombus~$R_2$ with nonedge~$e p_2$, so~$V(R_1) \cap V(R_2) = \{e\}$. Finally, add the edge~$p_1p_2$. Figure~\ref{fig:spindles} shows several spindles. For~$k \geq 1$, the spindle~$\spindle{k}{k}$ is called the \defi{$k$-Moser spindle}, denoted by~$\MS{k}$. If~$k \leq l$, then~$\spindle{k}{l} \subseteq \MS{l}$.

\begin{figure}[t]
    \centering
    \includegraphics{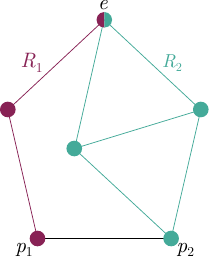}\qquad
    \includegraphics{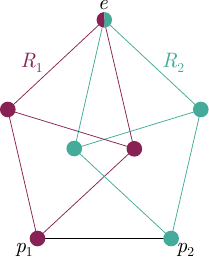}\qquad
    \includegraphics{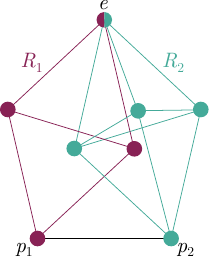}
    \caption{From left to right:~$\spindle{1}{2}$,~$\MS{2}$, and~$\spindle{2}{3}$. The~$R_i$ are indicated with their respective colors.}
    \label{fig:spindles}
\end{figure}

The~$(k,l)$-spindle is an anti-triangle-free graph of order~$k + l + 3$. The Moser spindle in particular is a well-studied object. For example, the spindle~$\MS{n-1}$ was already pointed out by Bezdek and Lángi~\cite{Bezdek1999AlmostSd-1} as a $t$-almost-equiangular set for~$t$ close to~1. However, they did not attempt to calculate for which~$t$ the graph is realizable, and did not consider the case of negative~$t$ or~$k \neq l$.  This is done in the following theorem.

\begin{theorem}%
\label{thm:RealizabilityOfSpindles}
    If~$k \geq 1$ and~$i \in \{1,2,3\}$, then all roots~$t_{k,i}$ of the polynomial
    \[
    8k^2t^3 - (k^2 -10k + 1)t^2 - 2(k - 1)t - 1
    \]
    with respect to~$t$ are real and can be ordered such that~$t_{k, 1} \leq t_{k, 2} < 0 < t_{k, 3}$.  The~$(k,l)$-spindle is~$(n,t)$-realizable if and only if~$k, l \leq n - 1$ and~$t \in [-1, 1)$ satisfy
    \begin{align}
        \label{eq:spindle-a} &t = (-1/4)(1 \pm \sqrt{5}) &&\text{ if } n = 2\text{ and }k = l = 1,\\
        \label{eq:spindle-b} &t_{k, 1} \leq t \leq t_{k, 2}\text{\quad or\quad} t_{k , 3} \leq t &&\text{ if } n > 2 \text{ and }k = l = n - 1,\\
        \label{eq:spindle-c} &t_{k,1} \leq t&&\text{ if }n > 2\text{ and } k = l < n - 1,\text{ or}\\
        \label{eq:spindle-d} &-1/l < t &&\text{ if } n > 2\text{ and }k < l \leq n - 1.
    \end{align}
\end{theorem}

The following simple lemma does a lot of the work in the proof of Theorem~\ref{thm:RealizabilityOfSpindles}.
\begin{lemma}
    If~$S_1$ and~$S_2$ are subsets of~$S^{n-1}$ that are invariant under the subgroup of~$\ortho(n)$ that stabilizes a point~$e$ and if $\inf{\{\, x \cdot y : x \in S_1, y \in S_2\, \}}$ is attained by points~$p_1 \in S_1$ and~$p_2 \in S_2$, then~$e$, $p_1$, and~$p_2$ lie on a great circle~$C$. Moreover, if~$f \in C$ is orthogonal to~$e$ and if~$p_1$, $p_2 \neq \pm e$, then~$f\cdot p_1$ and~$f \cdot p_2$ have opposite signs.
    \label{lem:MaxDistanceWhenOpposite}
\end{lemma}
\begin{proof}
    If~$p_1$ or~$p_2$ is~$\pm e$, then the result is clear.  So assume~$p_1$, $p_2 \neq \pm e$.

    Let~$U \coloneqq \Span\{ e, p_1 \}$ and let~$f$ be a unit vector in~$U$ orthogonal to~$e$ such that~$f \cdot p_1 > 0$. Write~$p_1 = \alpha e + \beta f$ and~$p_2 = \lambda e  + \kappa f + w$ with~$w \in U^{\perp}$, so~$|\kappa| \leq \sqrt{1 - \lambda^2}$.  By invariance under the stabilizer of~$e$, any point~$p_2'$ on the sphere with~$p_2 \cdot e = p_2'\cdot e$ is also in~$S_2$. Let~$p_2' = \lambda e - \sqrt{1-\lambda^2}f \in S_2$. Then
    \[
        p_1 \cdot p_2' = \alpha \lambda - \beta \sqrt{1 - \lambda^2} \leq \alpha \lambda + \beta \kappa = p_1 \cdot p_2.
    \]    
    It follows that~$w = 0$ and that~$f \cdot p_2 = \kappa < 0$, as wanted.
\end{proof}

\begin{proof}[Proof of Theorem~\ref{thm:RealizabilityOfSpindles}]
Let~$n \geq 2$ and~$1 \leq k \leq l \leq n - 1$ be integers and let~$t \in (- 1 / l, 1)$. A~$(k, l)$-spindle contains the union of a~$k$- and an~$l$-rhombus that intersect in a single point. Let~$R_1$ be the~$k$-rhombus with unique nonedge~$e p_1$ and~$R_2$ the~$l$-rhombus with unique nonedge~$e p_2$, so~$V(R_1) \cap V(R_2) = \{ e \}$. A necessary and sufficient condition for realizability is that there are realizations of~$R_1$ and~$R_2$ such that~$p_1 \cdot p_2 = t$.

Let~$S_i$ be the set of all possible images of~$p_i$ under $(n, t)$-realizations of~$R_i$ that map~$e$ to the north pole~$(1, 0, \ldots, 0)$, that is,
\[
    S_i \coloneqq \{\, f(p_i) : f \text{ is an~$(n, t)$-realization of } R_i\text{ such that }f(e) = (1, 0, \ldots, 0)\, \}.
\]
Let
\begin{equation}%
\label{eq:tau-formula}
    \tau_1 = \frac{2 k t^2}{(k - 1) t + 1} - 1\quad\text{and}\quad \tau_2 = \frac{2 l t^2}{(l - 1) t + 1} - 1.
\end{equation}
If~$n > 2$, then Lemma~\ref{lem:kRhombusRealizable} guarantees the existence of an~$(n, t)$-realization of~$R_1$ with~$e \cdot p_1 = \tau_1$. By rotating the realization,~$e$ can be placed at the north pole. If~$k < n - 1$, the lemma similarly guarantees the existence of an~$(n,t)$-realization of~$R_1$ with~$e \cdot p_1 = t'$ for all~$t'\in [\tau_1, 1)$ with~$e$ at the north pole. This goes through analogously for~$R_2$. Since the action of the stabilizer of~$e$ in~$\ortho(n)$ is transitive on the set of points~$p$ that have inner product~$t'$ with~$e$ for all~$t' \in [-1, 1]$, this shows that if~$k = n - 1$ or~$l = n - 1$, the corresponding~$S_i$ is
\[
    S_i = \{\, p \in S^{n-1} : e \cdot p = \tau_i\}
\]
and if~$k < n - 1$ or~$l < n - 1$, the corresponding~$S_i$ is
\[
    S_i = \{\, p \in S^{n-1} : e \cdot p \in [\tau_i, 1)\, \}.
\]
In particular, they are invariant under the stabilizer of~$e$ in~$\ortho(n)$.

Furthermore, if~$k \leq l \leq n-1$ and~$t > -1/l$, then~$\tau_1 \leq \tau_2$ for fixed~$t$, so that~$S_2 \subseteq S_1$.  It follows that there is~$\xi$ such that
\[
\{\, p \cdot q : p \in S_1,\ q \in S_2\,\} = [\xi, 1].
\]

Note that~$\xi$ is a function of~$k$, $l$, and~$t$.  Given~$n > 2$ and~$k$ and~$l$, it is then enough to find the values of~$t$ for which~$\xi \leq t$.  Let~$q_1 \in S_1$ and~$q_2 \in S_2$ be such that~$\xi = q_1 \cdot q_2$.  The goal is then to have~$q_1 \cdot q_2 \leq t$.  The following simple fact will be useful:
\begin{equation}%
\label{ass:MaximumDistanceOnUnitCircle}
\assert{If~$S_1$ and~$S_2$ are arcs of the unit circle~$S^1$ such that the infimum $\inf{\{\, x \cdot y : x \in S_1, y \in S_2\, \}}$ is attained, then the infimum is attained by an antipodal pair or by endpoints of the arcs.}
\end{equation}
By Lemma~\ref{lem:MaxDistanceWhenOpposite} it can be assumed that~$q_1$,~$q_2$, and~$e$ all lie on the same great circle~$C$.  By~\eqref{ass:MaximumDistanceOnUnitCircle}, either the~$q_i$ are endpoints of~$S_i \cap C$ or they are antipodal.

If the~$q_i$ are endpoints, then~$e \cdot q_i = \tau_i$.  Using Lemma~\ref{lem:MaxDistanceWhenOpposite} again gives
\[
    E \coloneqq q_1 \cdot q_2 = \tau_1 \tau_2 - \sqrt{1 - \tau_{\smash{1}}^{\smash{2}}} \sqrt{1 - \smash{\tau_2}^{\smash{2}}}.
\]
Hence, in this case the spindle is $(n, t)$-realizable if and only if~$E \leq t$.

The~$q_i$ are antipodal only if~$\tau_1 \leq -\tau_2$.  In this case, the spindle is $(n, t)$-realizable.  This gives necessary and sufficient conditions for realizability in the~$n > 2$ case.

If~$n = 2$, then~$k = l = 1$.  The sets~$S_i$ then each contain only two choices for~$q_i$ such that~$e \cdot q_i = \tau_i$. A necessary and sufficient condition for realizability is then that~$q_1 \cdot q_2 = t$.

To summarize, necessary and sufficient conditions for $(n, t)$-realizability of the $(k,l)$-spindle are:
\begin{enumerate}
\item[(i)] $E = t$\quad if~$n = 2$;

\item[(ii)] $E \leq t$\quad if~$n > 2$ and~$k = l = n - 1$;

\item[(iii)] $E \leq t$\quad or\quad $\tau_1 \leq -\tau_2$\quad otherwise.
\end{enumerate}

Recall from~\eqref{eq:tau-formula} that the~$\tau_i$ are functions of~$k$, $l$, and~$t$, and hence so is~$E$.  The goal is now to determine, for each case above, the values of~$t$ for which the conditions hold.

In most of the cases below, one has~$k = l$.  Then~$\tau_1 = \tau_2 \eqqcolon \tau$, and so
\[
    E  = 2\tau^2 - 1.
\]
Plug~\eqref{eq:tau-formula} into the right-hand side above to see that~$E \leq t$ if and only if
\begin{equation}
    8k^2t^3 - (k^2 -10k + 1)t^2 - 2(k - 1)t - 1
    \label{eq:MoserSpindlePolynomial}
\end{equation}
is nonnegative, with equality when~$t$ is a root of the polynomial.  In what follows, this and other polynomials considered are seen as polynomials on~$t$ only, that is,~$k$ is fixed.
\medbreak

\noindent
{\sc Case}~(i).  If~$n = 2$, then~$k = l = 1$, and there are only two values of~$t$ for which~$\MS{1}$ is realizable. To see this, factor the polynomial~\eqref{eq:MoserSpindlePolynomial} as
\[
    8t^3 + 8 t^2 - 1 = (2t + 1) (4t^2 + 2t -1).
\]
For the root~$t = -1 / 2$, the points~$p_1$ and~$p_2$ coincide with other points in the spindle. The other roots are~$t = -(1/4)(1 \pm \sqrt{5})$.  These inner products correspond to the pentagon and pentagram.  This gives~\eqref{eq:spindle-a}.
\medbreak

\noindent
{\sc Case}~(ii). If~$n > 2$ and~$k = l = n - 1$, then~(ii) is satisfied if and only if the polynomial~\eqref{eq:MoserSpindlePolynomial} has a nonnegative value at~$t$.  Its discriminant is positive so it only has real roots. Denote them by~$t_{k,1} \leq t_{k,2} \leq t_{k,3}$.  The constant and linear terms are negative, so~$t_{k,1} \leq t_{k,2} < 0 < t_{k,3}$. At~$t = 0$ the polynomial is negative, thus the polynomial must be nonnegative for~$t_{k,1} \leq t \leq t_{k,2}$ and~$t \geq t_{k,3}$. So~$\MS{n - 1}$ is realizable if and only if~$t_{k,1} \leq t \leq t_{k,2}$ or~$t \geq t_{k,3}$.  This establishes~\eqref{eq:spindle-b}.
\medbreak

\noindent
{\sc Case}~(iii).  It remains to consider~$n > 2$ and~$l < n - 1$.  The discussion splits into two cases: (a).~$k = l$ and (b).~$k < l$.
\medbreak

\noindent
{\sl Case (a).} If~$k = l < n - 1$, either one of the conditions in~(iii) has to be satisfied. The first one is again equivalent to finding~$t$ such that the polynomial~\eqref{eq:MoserSpindlePolynomial} is nonnegative, and so a sufficient condition for realizability is~$t_{k,1} \leq t \leq t_{k,2}$ or $t \geq t_{k,3}$.

The second condition is~$\tau_1 \leq -\tau_2$.  Since~$\tau_1 = \tau_2 \eqqcolon \tau$ one has~$\tau \leq 0$.  From~\eqref{eq:tau-formula}, this happens if and only if~$g \coloneqq 2kt^2 - (k - 1)t - 1 \leq 0$. This polynomial has a positive and a negative root and is negative at~$0$. At both roots,~\eqref{eq:MoserSpindlePolynomial} is positive. This can be tested by taking the remainder of~\eqref{eq:MoserSpindlePolynomial} after division by~$g$, and testing it at a convenient value smaller then the smallest root of~$g$ (for example~$t = -1/k$), since the remainder is linear and increasing in~$t$. So~$\MS{k}$ with~$k < n - 1$ is realizable if and only if~$t_{k,1} \leq t$.  This establishes~\eqref{eq:spindle-c}.
\medbreak

\noindent
{\sl Case (b).} The final case is~$n > 2$ and~$k < l \leq n - 1$.  It turns out that it suffices to consider the case~$l = k + 1$, as will be seen later.  

So assume~$l = k + 1$.  Let
\[
    f \coloneqq 8k^2 \left( k - 1 \right) t^4 - \left( k^3 - 19k^2 + 8k + 4 \right) t^3 - k \left( 3k - 14 \right) t^2 - 3 \left( k - 1 \right) t - 1.
\]
The inequality~$E \leq t$ is satisfied if and only if~$f \geq 0$.

If~$k=1$, then~$f$ is of degree~$3$.  Computing its roots, one gets conditions for the inequality above to be satisfied, obtaining a set of values of~$t$ for which the spindle is realizable.  Similarly, the condition~$\tau_1 \leq -\tau_2$ is satisfied if and only if~$t^3 + 3t^2 - t - 1 \leq 0$.  This gives another set of values of~$t$ for which the spindle is realizable.  Taking the union of both sets, one gets the condition~$t > -1/2 = -1/l$ for realizability.

If~$k > 1$, then~$f$ has degree~$4$ and its discriminant is negative, so it has exactly two real roots~$f_1 \leq f_2$. At~$t = 0$ it is negative and at~$t = 1$ and~$t = - 1 / (k + 1)$ it is positive, hence~$-1/(k+1) < f_1 < 0 < f_2 < 1$ and~$f \geq 0$ for~$-1/(k + 1) < t \leq f_1$ and~$f_2 \leq t < 1$.

The condition~$\tau_1 \leq -\tau_2$ is equivalent to the condition
\[
    g \coloneqq \left(2 k^2-1\right) t^3 - \left(k^2 - 3k - 1\right) t^2 - (2k - 1)t - 1 \leq 0.
\]
By an analysis similar as before, this polynomial has three real roots~$g_1 \leq g_2 < 0 < g_3$. It is negative at~$0$, so it is nonpositive for all~$t$ such that~$t \leq g_1$ or~$g_2 \leq t \leq g_3$. The next objective is to show that in fact~$g_2 \leq f_1 \leq f_2 \leq g_3$, so that the result follows; see Figure~\ref{fig:order-of-roots}.
\begin{figure}[tb]
    \centerfloat
    \includegraphics{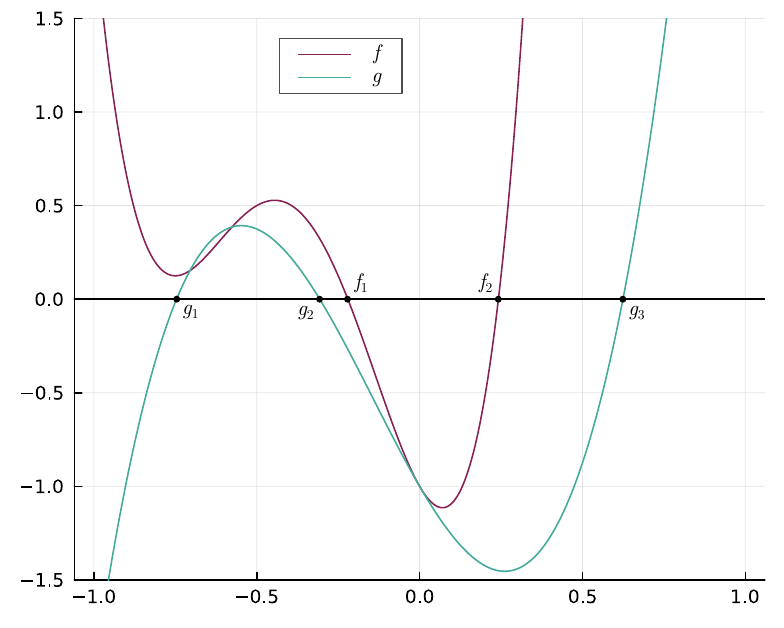}
    \caption{A plot of~$f$ and~$g$ for~$k = 2$. The horizontal axis is the inner product~$t$. Clearly~$-1/2 < g_2 \leq f_1 < 0 < f_2 \leq g_3$. If~$n > 2$, then~$\spindle{1}{2}$ is~$(n, t)$-realizable if and only if~$t > -1/2$, and~$f \geq 0$ or~$g \leq 0$. This plot shows that it is~$(n, t)$-realizable if and only if~$t > -1/2$.}
    \label{fig:order-of-roots}
\end{figure}

To determine the order of the roots~$f_1$,~$f_2$,~$g_1$,~$g_2$, and~$g_3$, take the remainder~$r$ of~$f$ after division by~$g$.  The remainder has degree~2 and has two real roots; denote the roots of~$r$ by~$r_1 \leq r_2$. Then~$f$ is nonnegative at a~$g_i$ if and only if~$r$ is. Both roots of~$r$ are negative for any~$k \geq 2$. Moreover,~$g$ is positive at~$r_1$ and~$r_2$, so they lie between~$g_1$ and~$g_2$. The coefficient of the quadratic term of~$r$ is positive, so it has a global minimum, meaning it is positive for all~$t > r_2 > g_1$, so~$f$ is positive at~$g_2$ and~$g_3$. This determines the order of the roots~$g_2 \leq f_1 < 0 < f_2 \leq g_3$. The spindle is realizable if~$-1/(k + 1) < t \leq f_1$,~$f_2 \leq t \leq 1$ and~$g_2 \leq t \leq g_3$, so putting all of this together,~$\spindle{k}{k+1}$ is realizable if and only if~$-1/(k + 1) < t < 1$.

From the case~$l = k + 1$ all other cases follow. Indeed,~$\spindle{k}{l}$ with~$k < l - 1$ is a subgraph of~$\spindle{l - 1}{l}$, and so a sufficient condition for realizability is~$t > - 1/l$, which was already seen to be necessary. This settles~\eqref{eq:spindle-d}.
\end{proof}

\subsection{Some results on non-realizability}
\label{subsec:NonRealizableGraphs}

In order to classify almost-equiangular sets in low dimension, it is necessary to show that given anti-triangle-free graphs are not $(n, t)$-realizable for certain~$n$ and~$t$.

\subsubsection{The extended rhombus}

Let~$t = - 1/ n$ and take two~$(n - 1)$-rhombi,~$R_1$ and~$R_2$, that intersect in an induced subgraph~$\Sigma$ isomorphic to~$\comp{n}$ (see Figure~\ref{fig:ExtendedRhombus}). Call this graph an \defi{extended~$(n-1)$-rhombus}. Let~$e$ and~$p$ be the endpoints of the unique nonedge of~$R_1$ with~$p \in \Sigma$. Let~$f$ be the endpoint of the nonedge of~$R_2$ not contained in~$\Sigma$.

If~$t = -1/n$ and~$k =  n - 1$, then by Lemma~\ref{lem:kRhombusRealizable},~$e \cdot p = -1 / n$ in any realization of~$R_1$. So a realization of~$R_1$ actually forms an~$n$-simplex, and analogously the same holds for~$R_2$. But then~$e$ and~$f$ are uniquely determined by~$\Sigma$, and must coincide, hence the extended $(n-1)$-rhombus is not $(n, -1/n)$-realizable.

\begin{figure}[bt]
    \centering
    \includegraphics{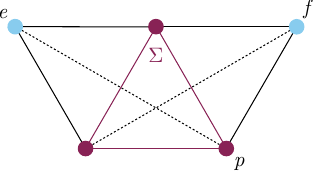}\qquad
    \includegraphics{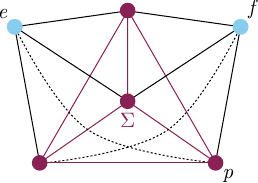}
    \caption{The extended~2-rhombus on the left and the extended~3-rhombus on the right. The dotted lines are edges that follow from Lemma~\ref{lem:kRhombusRealizable}, forcing~$f$ to coincide with~$e$.}
    \label{fig:ExtendedRhombus}
\end{figure}

\subsubsection{The complement of the split $k$-cycle}

\begin{figure}[bt]
    \centering
    \includegraphics{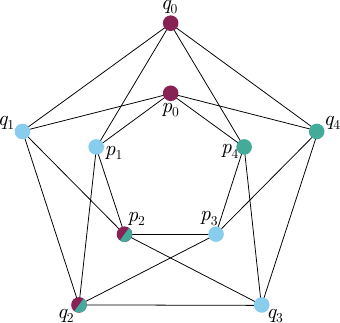}\qquad
    \includegraphics{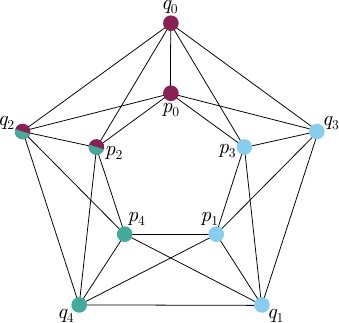}
    \caption{On the left the graph~$W_5$ with~$\Sigma_0$ and~$\Sigma_2$ indicated by color. In this case,~$T_2 = \{p_2, q_2\}$. On the right~$\overline{W_5}$ after rearranging the vertices, with~$\Sigma_0$ and~$\Sigma_2$ colored as well.  The similarity between the two graphs is incidental for~$k = 5$.}
    \label{fig:W5}
\end{figure}

Let~$k \geq 4$. The \defi{split $k$-cycle} is the graph~$W_k$ on vertices~$p_0,\ldots, p_{k-1}, q_0,\ldots, q_{k-1}$ in which the neighborhood of both~$p_i$ and~$q_i$ is~$\{p_{i-1},q_{i-1},p_{i+1},q_{i+1}\}$ with all indices modulo~$k$ (see Figure~\ref{fig:W5}). It is obtained from a~$k$-cycle by splitting each vertex. Deaett proved~\cite[Theorem 4.11]{Deaett2011TheGraph} that the graph~$\overline{W_n}$, the complement of~$W_n$, is~$(n,0)$-realizable.

For even~$k$, the graph~$W_k$ is bipartite with parts of size~$k$, since the set of all even-indexed points is independent and so is its complement. This means that~$\overline{W_k}$ is a union of two~$(k - 1)$-simplices with some extra edges and therefore does not give a new construction.

For~$k = 5$, the graph~$\overline{W_5}$ is~$(5,0)$-realizable (see Figure~\ref{fig:W5}).  It is the smallest example of an optimal~$(n, 0)$-realizable anti-triangle-free graph that is not a union of two~$(n - 1)$-simplices~\cite{Deaett2011TheGraph}. Balko, Pór, Scheucher, Swanepoel, and Valtr showed~\cite[Theorem 2]{Balko2020Almost-EquidistantSets} that~$\overline{W_5}$ cannot be embedded in~$\R^3$ so that adjacent vertices are at distance~$1$.  Since there are $(4, -1/4)$-realizable graphs of order~$10$, a priori~$\overline{W_5}$ could be $(4, -1 / 4)$-realizable. It turns out, however, that~$\overline{W_k}$ with odd~$k \geq 5$ is not $(k-1, t)$-realizable for any negative~$t$.

Indeed, take~$W_{k}$ with odd~$k \geq 5$. The optimization bound (Theorem~\ref{thm:interpolation}) shows that if~$n < k$ and~$t \in [-1, 0]$, then the maximum cardinality of a $t$-almost-equiangular set on~$S^{n-1}$ is~$\leq 2(n+1)$, with equality only at~$t = -1/n$. Since~$\overline{W_k}$ has order~$2k \geq 2(n+1)$, it can only be $(n,t)$-realizable for~$n < k$ when~$n = k - 1$ and~$t = -1/n$.

Hence the goal is to show that~$\overline{W_k}$ is not $(n, -1/n)$-realizable with~$n = k - 1$.  So assume that~$\overline{W_k}$ is realizable.

In what follows, indices are taken modulo~$k$.  Let~$\Sigma_i$ be the set of all vertices~$p_{i+2j}$ and~$q_{i+2j}$ for~$0 \leq j \leq (k-3)/2$ and set~$T_i\coloneqq \Sigma_{i-2} \cap \Sigma_i$ (see Figure~\ref{fig:W5}).

The~$\Sigma_i$ are independent in~$W_k$ and so form~$(k - 2)$-simplices in a realization of~$\overline{W_k}$. Take the sets~$\Sigma_0$ and~$\Sigma_2$. Then~$\Sigma_0 \setminus T_2$ and~$\Sigma_2 \setminus T_2$ both consist of two points that lie in the intersection of hyperplanes defined by the equations~$l \cdot x = t$ for all~$l \in T_2$. The realization of~$T_2$ is a~$(k - 4)$-simplex, so by Section~\ref{sec:simplex},~$T_2$ consists of~$k - 3$ linearly independent vectors and the dimension of the intersection of these hyperplanes is~2. Therefore,~$p_0$,~$q_0$,~$p_{k-1}$, and~$q_{k-1}$ are coplanar and lie on a circle~$C_1$. Repeat this for~$\Sigma_1$ and~$\Sigma_3$ to see that~$p_0$,~$q_0$,~$p_1$, and~$q_1$ are also coplanar and lie on a circle~$C_2$.

Since~$K \coloneqq \{p_1, q_1, p_{k-1}, q_{k-1}\}$ is a clique in~$\overline{W_k}$, it defines a regular tetrahedron, hence its affine span is 3-dimensional, and the circles~$C_1$ and~$C_2$ are distinct. Denote the circumsphere of~$K$ by~$S$, which is a~2-sphere. The affine span of~$\{p_0, q_0, p_1, q_1, p_{k-1}, q_{k-1}\}$ is also 3-dimensional, since these points lie on two distinct planes intersecting on a line. Then~$p_0$, $q_0 \in \Aff K$. By uniqueness of the circumsphere of a simplex this means~$p_0$ and~$q_0$ also lie on~$S$.

Since~$\Sigma_0$ can be completed to a regular~$(k - 1)$-simplex for~$t = -1/(k - 1)$ by adding a point on~$z \in C_1$, it follows that~$C_1$ is a circumcircle of a regular triangle on~$S$ whose vertices are~$z$, $p_{k-1}$, and~$q_{k-1}$. However, there are only two such regular triangles on~$S$, namely~$\{p_{k-1}, q_{k-1}, p_1\}$ and~$\{p_{k-1}, q_{k-1}, q_1\}$. So~$C_1$ contains~$p_0$,~$q_0$,~$p_{k-1}$,~$q_{k-1}$ and either~$p_1$ or~$q_1$. By a similar argument,~$C_2$ contains~$p_0$,~$q_0$,~$p_1$,~$q_1$ and either~$p_{k - 1}$ or~$q_{k - 1}$. Then~$C_1$ intersects~$C_2$ in at least four points, a contradiction.

 \section{Maximum obtuse almost-equiangular sets}
 \label{sec:properties-and-uniqueness}

Theorem~\ref{thm:maximal-almost-equiangular-sets} below establishes that~$\alpha(n, t) \leq 2(n+1)$ for all~$t \leq 0$, with equality only for~$t = -1/n$.  This motivates calling a $(-1/n)$-almost-equiangular set with~$2(n+1)$ points a \defi{maximum obtuse almost-equiangular set}.

The proof of Theorem~\ref{thm:maximal-almost-equiangular-sets} follows a spectral analysis of matrices associated to the Gram matrix of such a set, done by Rosenfeld~\cite{Rosenfeld1991AlmostEd} and Bezdek and Lángi~\cite{Bezdek1999AlmostSd-1}. Further analysis of these matrices gives useful properties of maximum obtuse almost-equiangular sets; they turn out to be spherical $2$-designs, and are in bijection with certain symmetric orthogonal matrices.  

Finally, this leads to a proof that the only maximum obtuse almost-equiangular set is the double regular $n$-simplex for~$n = 2$, \dots,~$5$.

\subsection{The spectral analysis}

Bezdek and Lángi prove in~\cite{Bezdek1999AlmostSd-1} that a $t$-almost-equiangular subset of~$S^{n-1}$ with~$t \leq 0$ cannot have more that~$2(n+1)$ points,  by  analyzing the eigenvalues of a certain matrix related to the  Gram matrix of the set. Their method is revisited here to strengthen their result as 
follows.

\begin{theorem}\label{thm:maximal-almost-equiangular-sets}
If~$t \in [-1, 0]$, then~$\alpha(n, t) \leq 2(n+1)$, with equality only at~$t = -1/n$. The Gram matrix of a maximum obtuse almost-equiangular set has rank~$n$, its only nonzero eigenvalue is $2(1+1/n)$, and the all-ones vector~$e$ is in its kernel. In particular, the barycenter of a maximum obtuse almost-equiangular set is~$0$.
\end{theorem}

\begin{proof}
Following~\cite{Bezdek1999AlmostSd-1}, let $U$ be the Gram matrix of a $t$-almost-equiangular subset of $S^{n-1}$ of cardinality $N$,  let 
$C=U-t J$,  and $B=U-t J-(1-t)I$,  where $J$ is the all-ones matrix and~$I$ is the identity matrix. The diagonal coefficients of~$B$ are~$0$,  hence~$\trace B = 0$.  The coefficients of~$B$ corresponding to pairs of points with inner product~$t$  are equal to~$0$, hence the set being almost equiangular translates to $B_{ij}B_{jk}B_{ki}=0$ for all $1\leq i,j,k\leq N$, whence~$\trace(B^3)=0$.

These two properties give rise to equations for the eigenvalues of $B$.  Because $\rank C \leq n+1$,  the matrix $B$ has at least  $N-(n+1)$ eigenvalues equal to $-(1-t)$.  If $\lambda_1$, \dots,~$\lambda_{n+1}$ denote the remaining ones, then
\[
\sum_{i=1}^{n+1} \lambda_i = (N-n-1)(1-t)\quad\text{and}\quad
\sum_{i=1}^ {n+1} \lambda_i^3 =(N-n-1)(1-t)^3.
\]

Since~$t < 0$, the matrix~$C$ is positive semidefinite, and so the smallest eigenvalue of~$B$ is~$-(1-t)$.  Hence if $y_i=\lambda_i/(1-t)$, then~$y_i \geq -1 > -\sqrt{3}$ and the problem
\begin{equation}%
\label{opt:eigs}
\begin{optprob}
z^* \coloneqq \max&\sum_{i=1}^{n+1} y_i\\
&\sum_{i=1}^{n+1} y_i - y_i^3 = 0,\\
&y_i \geq -\sqrt{3}\quad\text{for~$i = 1$, \dots,~$n + 1$}
\end{optprob}
\end{equation}
gives an upper bound for~$N - n - 1$.

Let
\[
    L(y) \coloneqq \sum_{i=1}^{n+1} y_i + (1/2)\biggl(\sum_{i=1}^{n+1} y_i - y_i^3\biggr)
    = \sum_{i=1}^{n+1} (3/2) y_i - (1/2) y_i^3
\]
and
\begin{equation}%
\label{opt:lag}
    d^* \coloneqq \max\{\, L(y) : \text{$y_i \geq -\sqrt{3}$ for all~$i$}\,\}.
\end{equation}
If~\eqref{opt:lag} has an optimal solution~$y^*$ that is feasible for~\eqref{opt:eigs}, then it is also optimal for~\eqref{opt:eigs}. Conversely, if~$z^* = d^*$ and~$y^*$ is optimal for~\eqref{opt:eigs} then it is optimal for~\eqref{opt:lag}.

In an optimal solution of~\eqref{opt:lag} all the~$y_i$ have the same value, namely
\[
    \max\{\, p(y) : y \geq -\sqrt{3}\, \},
\]
where~$p(y) \coloneqq (3/2)y - (1/2)y^3$.  A boundary and critical point analysis on~$p$ shows it has a unique maximum for~$y \geq -\sqrt{3}$ given by~$p(1) = 1$.

Therefore, the problem~\eqref{opt:lag} has a unique optimal solution~$y^*$ with~$y^*_i = 1$ for all~$i$, and its optimal value is~$n + 1$. Since~$y^*$ is also feasible for~\eqref{opt:eigs}, it is its unique optimal solution with optimal value~$n + 1$. So, $N\leq 2(n+1)$, and equality holds if and only if the matrix $B$ has exactly  $n+1$ eigenvalues equal to $1-t$ and $n+1$ eigenvalues equal to $-(1-t)$.  It follows that if~$N = 2(n+1)$, then~$C$ has exactly one nonzero eigenvalue, namely~$2(1-t)$ with multiplicity~$n+1$.

Assume that~$N = 2(n+1)$, so the set attains the maximum cardinality. Then, the all-ones vector~$e$ is in the kernel of~$U$, and~$U$ has rank~$n$. Indeed,~$\rank U \leq n < n+1 = \rank C$, and since~$C = U - tJ$ it follows that~$e$ is not in the column space of~$U$, so~$e$ is in the column space of~$C$. The column space~$E$ of~$C$ is the eigenspace of~$C$ with eigenvalue~$2(1-t)$. Let~$S \subseteq E$ be the orthogonal complement to the span of~$e$ in~$E$. If~$x \in S$, then
\[
    Ux = Cx + tJx = Cx = 2(1-t)x,
\]
hence~$x$ is an eigenvector of~$U$. Since~$\rank U < \rank C$ it follows that~$S$ is the only eigenspace of~$U$ with nonzero eigenvalue. Hence,~$Ue = 0$ and~$U$ has rank~$n$.

The equation~$Ue = 0$ means that the barycenter of the set is~$0$. Moreover, from~$0 = Ue = (C + tJ)e = (2(1-t) + 2(n+1)t)e$ it follows that~$t = -1/n$.
\end{proof}

By the continuous dependence of eigenvalues on the coordinates of a matrix, the bound~$\alpha(n, t) \leq 2(n+1)$ can be extended to~$[-1, \varepsilon(n))$, where~$\varepsilon(n)$ is some (small) positive number depending on~$n$, something Bezdek and Lángi already showed. However, it is not true that this bound is global on~$t \in [-1,1)$, as a construction of Larman and Rogers~\cite{Larman1972TheSpace} shows. Namely, let~$n=5$ and~$S$ be the set of vertices of the cube~$[-1, 1]^5$ that have an odd number of positive signs. Then~$|S| = 16$ with vectors of norm~$\sqrt{5}$.  Rescaling by~$\sqrt{5}$ gives a $(1/5)$-almost-equiangular set on~$S^4$ of cardinality~$16$.

The proof of Theorem~\ref{thm:maximal-almost-equiangular-sets} also links the maximum obtuse almost-equiangular sets to the theory of spherical designs (see the survey by Bannai and Bannai~\cite{Bannai2009ASpheres} for more on spherical designs).

\begin{proof}[Proof of Theorem~\ref{cor:optimal-construction-are-2-designs}]
    According to Theorem~\ref{thm:maximal-almost-equiangular-sets}, $\sum_{i=1}^{2(n+1)} x_i=0$, the Gram matrix~$U$ of~$S$ satisfies $U^2=2(1+1/n)U$, and~$U$ has rank~$n$.  Moreover, the identity $U^2=2(1+1/n)U$ translates to
\begin{equation*}
\sum_{k=1}^{2(n+1)}  (x_i\cdot x_k)(x_k \cdot x_j) =2(1+1/n)(x_i\cdot x_j) \quad\text{for all }1\leq i,j\leq 2(n+1).
\end{equation*}
By linearity, $x_i$ and $x_j$ can be replaced by any vector of $\R^n$. In particular, for all $u\in S^{n-1}$, 
\begin{equation*}
\sum_{k=1}^{2(n+1)}  (u\cdot x_k)^2 =2(1+1/n).
\end{equation*}
This identity, together with $\sum_{i=1}^{2(n+1)}x_i=0$, characterizes the spherical designs of strength $2$.  (See~\cite[Theorem~2.2]{Bannai2009ASpheres}, but note that in property (6) of this theorem the first appearance of the exponent~$k$ is wrong and should be~$2k$.)
\end{proof}

\subsection{Relation to orthogonal matrices}

The union of two vertex-disjoint regular $n$-simplices, called a double regular
$n$-simplex, is an example of a maximum obtuse almost-equiangular set.  A
natural question is whether this construction is unique.  The affirmative answer
for~$n \leq 5$ is established in Theorem~\ref{thm:double-simplex-uniqueness}.
Theorem~\ref{thm:O-matrix} works towards this proof, and is interesting by
itself.

\begin{proof}[Proof of Theorem~\ref{thm:O-matrix}]
With the same notation as in the proof of Theorem~\ref{thm:maximal-almost-equiangular-sets}, let~$B$ denote the matrix associated to a maximum obtuse almost-equiangular set of unit vectors.  The matrix~$B$ has only two eigenvalues, namely~$\pm (1+1/n)$, and hence satisfies $B^2=(1+1/n)^2I$.  Moreover, $Be=(1+1/n)e$.  Let $O \coloneqq (1+1/n)^{-1}B$; it is clear from the properties of $B$ that~$O$ is symmetric and orthogonal and that it satisfies the conditions~(1)--(3).

Conversely, given a symmetric and orthogonal matrix~$O$ satisfying (1)--(3), let
\[
    U \coloneqq (1+1/n)O-(1/n)J+(1+1/n)I
\] 
and let $E_{\pm1}$ be the eigenspaces of $O$ associated with the two eigenvalues $\pm 1$. Both of them have dimension $n+1$ because $\trace(O)=0$ due to (2). 
The kernel of $U$ is the subspace $E_{-1}\oplus \R e$ of dimension $n+2$; its orthogonal complement is the eigenspace of $U$ associated to the eigenvalue $2(1+1/n)$. So $U$ is the Gram matrix of a set of $2(n+1)$ unit vectors in $\R^n$. Condition (3) ensures that this set is $(-1/n)$-almost equiangular.
\end{proof}

Any $t$-almost-equidistant set in~$S^{n-1}$ with~$t \leq 0$ can be lifted to an almost-orthogonal set on~$S^n$~\cite{Polyanskii2017OnII}. Since~$\alpha(n+1,0) = 2(n+1)$, every maximum obtuse almost-equidistant set gives a maximum almost-orthogonal set in this way. Since~$\overline{W_5}$ is $(5,0)$-realizable but not $(4,-1/4)$-realizable (see Section~\ref{subsec:NonRealizableGraphs}), the converse is not the case. Deaett characterized the maximum almost-orthogonal sets by a statement similar to Theorem~\ref{thm:O-matrix}; it differs only by the eigenvector condition~(1). Hence, the eigenvector condition distinguishes between those maximum almost-orthogonal sets that show this form of descent, and those that do not.

\subsection{The distance graph of a maximum obtuse almost-equiangular set}

A graph is \defi{quadrangular} if no two vertices have exactly one neighbor in common.

\begin{lemma}\label{lem:props-maximum-sets}
The following properties hold for the distance graph $G$ of a maximum obtuse almost-equiangular subset~$S$ of $S^{n-1}$.
\begin{enumerate}
\item If $G$ contains a $\comp{n+1}$, then~$S$ is a double regular $n$-simplex.

\item The graph~$\overline{G}$ is quadrangular.

\item The degree of a vertex in~$\overline{G}$ lies between $1$ and $n+1$.  
 If there is a vertex of degree~$n+1$ in~$\overline{G}$, then~$S$ is a double regular $n$-simplex.
If there is a vertex~$x$ with exactly one neighbor~$y$ in~$\overline{G}$, then~$G[S \setminus \{x, y\}]$ is the distance graph of a maximum obtuse almost-equiangular subset of~$S^{n-2}$.
\end{enumerate}
\end{lemma}

\begin{proof}
Let $O=(n / (n+1))U+(1 / (n+1))J-I$, where~$U$ is the Gram matrix of~$S$, be the matrix of Theorem~\ref{thm:O-matrix}.  The entries of~$O$ are equal to~$0$ on the diagonal and at pairs of vectors with inner product~$-1/n$,  so the adjacency matrix~$A$ of~$\overline{G}$ is such that~$A_{ij} = 0$ if~$O_{ij} = 0$ and~$A_{ij} = 1$ if~$O_{ij} \neq 0$.

If~$G$ contains a~$K_{n+1}$, then~$O$ is of the form~$O = \smallpmatrix{ 0& B\\B^\tp & D}$ where~$B$ and~$D$ are $(n+1)\times (n+1)$ matrices.  The condition~$O^2=I$ leads to~$BB^\tp=I$ and~$BD=0$. But then~$B$ is invertible and so~$D=0$, which proves~(1).

Property~(2) follows from the columns of~$O$ being pairwise orthogonal: if two vertices~$x_i$, $x_j$ share a single neighbor~$x_k$ in~$\overline{G}$,  then~$O_{ki}O_{kj} \neq 0$,  while~$O_{li}O_{lj} = 0$ for~$l \neq k$.  But then the columns~$i$ and~$j$ of~$O$ would not be orthogonal.

To prove~(3), note that~$\overline{G}$ is triangle free.  Let~$x$ be a vertex and let~$\overline{N}_x$ denote its set of neighbors in~$\overline{G}$.  Two vertices in~$\overline{N}_x$ cannot be adjacent in~$\overline{G}$, otherwise they would form a triangle with~$x$.  So~$\overline{N}_x$ is a clique in~$G$,  that is, it is a regular simplex,  which proves that the degree of~$x$ in~$\overline{G}$ is at most~$n+1$.  Moreover, if~$x$ has degree~$n+1$, then it follows from~(1) that~$G$ contains a~$\comp{n+1}$, and hence that~$S$ is a double regular $n$-simplex.

Next, given a vertex~$x$, let~$N_x$ be its neighborhood in~$G$.  All vertices in~$N_x$ have inner product~$-1/n$ with~$x$, and so lie in an affine hyperplane, and hence belong to an $(n-2)$-sphere~$C$.  By scaling and translating~$C$ via an affine transformation, it can be mapped to~$S^{n-2}$, and then~$N_x$ is mapped to a $t$-almost-equidistant set for some~$t \leq 0$.  It then follows from Theorem~\ref{thm:maximal-almost-equiangular-sets} that~$|N_x| \leq 2n$, and so the degree of~$x$ in~$\overline{G}$ is at least~$1$.  Moreover, if~$|N_x| = 2n$, then~$t = -1/(n-1)$.
\end{proof}

\subsection{Uniqueness of the double regular simplex.}

The goal in this section is to prove Theorem~\ref{thm:double-simplex-uniqueness}. For a given dimension~$n$, the theorem is false if there is an $(n, -1/n)$-realizable anti-triangle-free graph of order~$2(n+1)$ whose complement is not bipartite.  It turns out that, to prove the theorem, it is enough to show that such a graph whose complement contains a $5$-cycle is not realizable.

\begin{proof}[Proof of Theorem~\ref{thm:double-simplex-uniqueness}]
The distance graph of a maximum obtuse almost-equiangu\-lar set is anti-triangle free and, by Lemma~\ref{lem:props-maximum-sets}, has a quadrangular complement.  Moreover, if the set is not a double regular $n$-simplex, then the complement is not bipartite.  The goal of the proof is then to show that, if~$G$ is an anti-triangle free graph of order~$2(n+1)$ whose complement is quadrangular and nonbipartite, then~$G$ is not $(n, -1/n)$-realizable.  This is done below for~$2 \leq n \leq 5$.

Let~$G$ be an anti-triangle-free graph of order~$2(n+1)$ whose complement is quadrangular.  Say that~$\overline{G}$ does not contain odd cycles of length~$3$, $5$, \dots,~$2k-1$, but contains an odd cycle of length~$2k + 1$ with vertices~$p_0$, \dots,~$p_{2k}$.  Since~$\overline{G}$ is quadrangular, every pair of vertices~$p_i$, $p_{i+2}$, with indices taken modulo~$2k+1$, has at least two common neighbors.  One of the neighbors is~$p_{i+1}$; denote the other by~$q_{i+1}$.  Since~$\overline{G}$ does not contain odd cycles of length less than~$2k+1$, the vertices~$p_i$ and~$q_i$ must all be distinct, and so the order of~$G$ is at least~$2(2k+1)$, whence~$n \geq 2k$. This settles the case~$n = 3$.
\medbreak

\noindent
{\sc Dimensions~$4$ and~$5$.} It follows that, for~$n \leq 5$, if~$G$ is an anti-triangle-free graph of order~$2(n+1)$ whose complement is quadrangular and nonbipartite, then~$\overline{G}$ has an odd cycle of length~$5$, and since~$n \geq 2k$ as shown above, it is necessary that~$n \geq 4$.  So it suffices to show that such a graph~$G$ for~$n = 4$ and~$5$ is not $(n, -1/n)$-realizable.

\begin{figure}
\includegraphics{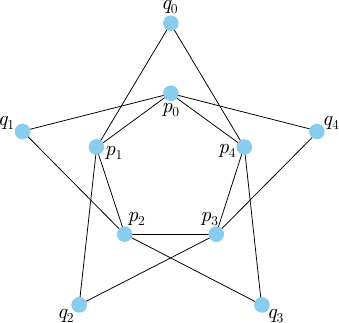}
\caption{If $\overline{G}$ contains a $5$-cycle,  it contains this subgraph.}
\label{fig:star}
\end{figure}

To this end, note that if~$p_0$, \dots,~$p_4$ is a $5$-cycle in~$\overline{G}$ and if~$q_0$, \dots,~$q_4$ are the common neighbors defined above, then~$\overline{G}$ has the graph in Figure~\ref{fig:star} as a subgraph.  Again since~$\overline{G}$ is quadrangular,  the pairs $p_i$, $q_{i+2}$ must have another common neighbor besides $p_{i+1}$.  If $n=4$,  there are no other vertices available,  so the only possibility is that~$q_0$, \dots,~$q_4$ is a cycle, that is,~$\overline{G}$ is isomorphic to~$W_5$ (see Figure~\ref{fig:W5}).  The graph $\overline{W_5}$ is not $(4,-1/4)$-realizable (see Section~\ref{subsec:NonRealizableGraphs}), so the proof is finished for~$n=4$.

The remaining case is~$n = 5$, for which~$G$ has order~$12$.  Call $x$, $y$ the two vertices of~$G$ other than the~$p_i$ and~$q_i$.   By an argument similar to the one above,~$\overline{G}$ contains as a subgraph either~$W_5$, as was the case for~$n = 4$, or, without loss of generality, the graph in Figure~\ref{fig:W5alt}.
\medbreak

\begin{figure}
\includegraphics{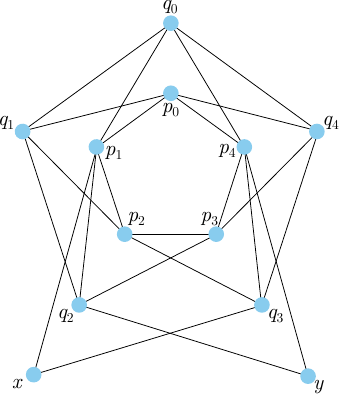}
\caption{When $n=5$,  the graph~$\overline{G}$ may contain this graph as a subgraph.}
\label{fig:W5alt}
\end{figure}

\noindent
\textsl{Dimension~$5$ and~$\overline{G}$ contains the graph of Figure~\ref{fig:W5alt}.} If $\overline{G}$ contains the graph of Figure~\ref{fig:W5alt}, then since~$\overline{G}$ is triangle free and~$x$ is adjacent to~$p_1$ and~$q_3$ in~$\overline{G}$, it must be that~$x$ is adjacent to~$p_0$, $q_0$, $p_2$, $q_2$, $p_4$, and~$q_4$ in~$G$.  The same reasoning for~$y$ shows that~$G$ contains as a subgraph the graph~$H_{12}$ from Figure~\ref{fig:H12}.  It will turn out that~$H_{12}$ is not $(5, -1/5)$-realizable.
\medbreak

\begin{figure}
\includegraphics{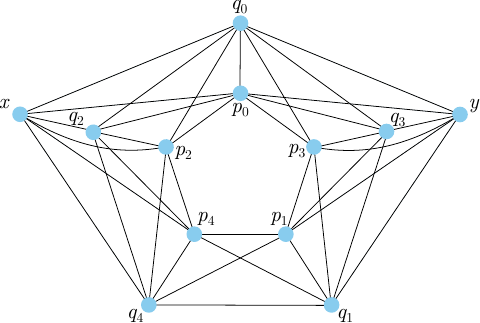}
\caption{The graph $H_{12}$ is not $(5,-1/5)$-realizable. }
\label{fig:H12}
\end{figure}

\noindent
\textsl{Dimension~$5$ and~$\overline{G}$ contains $W_5$.} If~$\overline{G}$ contains~$W_5$, then the graph~$G$ contains a subgraph isomorphic to~$H_{12}$ as well.  Indeed, in this case, the vertices~$x$ and~$y$ must be adjacent in~$\overline{G}$ to the subgraph~$W_5$, otherwise by~(3) of Lemma~\ref{lem:props-maximum-sets} the graph~$\overline{W_5}$ would be the distance graph of a $(-1/4)$-almost-equiangular set in~$S^3$ with~$10$ points that is not a regular double simplex, a contradiction.

If~$v$ is a vertex of~$W_5$ in~$\overline{G}$, then the neighborhood of~$v$ in~$W_5$ is an independent set, since~$\overline{G}$ is triangle free. The neighborhood forms a clique in~$G$; call it~$C_v$. If~$x$ is adjacent to~$v$ in~$\overline{G}$, again since~$\overline{G}$ is triangle free,~$x$ is adjacent to all vertices of~$C_v$ in~$G$.

Since~$x$ is adjacent in $\overline{G}$ to at least one vertex~$v$ of $W_5$, without loss of generality say~$x$ is adjacent to~$p_1$. Then,~$x$ is adjacent in~$G$ to~$C_{p_1} = \{p_0, q_0, p_2, q_2\}$. But then without loss of generality~$x$ is adjacent in~$G$ to~$\{p_0,q_0,p_2,q_2,p_4,q_4\}$.  Namely, if~$x$ is not adjacent to any of~$p_3$, $q_3$, $p_4$, and~$q_4$ in~$\overline{G}$, the statement follows immediately. Otherwise, if~$x$ is adjacent, say, to~$p_3$ in~$\overline{G}$, then~$x$ is adjacent in~$G$ to~$C_{p_1} \cup C_{p_3} = \{p_0, q_0, p_2, q_2, p_4, q_4\}$.

It remains to show that~$y$ is adjacent in~$G$ to all vertices in~$\{p_1,q_1,p_3,q_3\}$; applying the previous reasoning to~$y$ shows that if this is the case,~$y$ is adjacent to all vertices in either~$\{p_4,q_4,p_1,q_1,p_3,q_3\}$ or~$\{p_1,q_1,p_3,q_3, p_0,q_0\}$,  meaning that a subgraph isomorphic to~$H_{12}$ occurs in~$G$.

To prove that~$y$ is adjacent to all vertices in~$\{p_1, q_1, p_3, q_3\}$, consider the following. In order to arrive at a contradiction, assume~$y$ is adjacent to $p_1$ in~$\overline{G}$, again without loss of generality. Then~$y$ is connected in~$G$ to~$C_{p_1} = \{p_0,q_0,p_2,q_2\}$.  But~$x$ is also adjacent in~$G$ to these vertices,  and if~$G$ contains a~$\comp{6}$, then it is not $(5,-1/5)$-realizable ((1) of Lemma~\ref{lem:props-maximum-sets}), so $\{x, y\}$ is independent in~$G$.  Now the contradiction comes from the quadrangularity of~$\overline{G}$; indeed, if~$x$ and~$y$ are not adjacent in~$G$, then~$y$ is a common neighbor of~$x$ and~$p_1$ in~$\overline{G}$. But it is not possible that~$x$ and~$p_1$ have a second common neighbor because $x$ is not connected to any neighbor of~$p_1$ in~$\overline{G}$ other than~$y$.

To complete the proof, it remains to show that the graph $H_{12}$ is not $(5,-1/5)$-realizable. This is a specialization of a part of the proof of the nonrealizability of~$W_k$ from Section~\ref{subsec:NonRealizableGraphs}. In fact, the graph~$H_{12}$ is a subgraph of~$\overline{W_7}$, with two vertices and some edges removed. The removed edges play no role in the proof, and the two vertices only play a role for nonrealizability for~$n=6$, but for~$n=5$ they are superfluous.
\end{proof}
 
\section{Classification in dimensions 2 and 3}
\label{sec:low-dim}

Section~\ref{sec:realizability} gives exact conditions on the dimension~$n$ and inner product~$t$ for which simplices, rhombi, and spindles are $(n, t)$-realizable. For every integer~$m \geq 1$, dimension~$n$, and inner product~$t$, this gives sufficient conditions for the existence of $(n, t)$-realizable anti-triangle-free graphs of order~$m$.  These realizable graphs then give $t$-almost-equiangular sets of cardinality~$m$ in dimension~$n$.  In this section a converse result is obtained in low dimension: list all maximum-cardinality, almost-equiangular sets in~$S^{n-1}$, with~$n \leq 3$.

Say that an anti-triangle-free graph is \defi{minimal} if the removal of any edge results in a graph that is not anti-triangle free.  Given~$n$ and~$t$, say that an anti-triangle-free graph is \defi{$(n, t)$-optimal} if it is $(n, t)$-realizable and if it has order~$\alpha(n, t)$.  If a graph is the unique minimal $(n, t)$-optimal graph up to isomorphism, then it is called a \defi{unique optimal construction}.

Finding all minimal $(n, t)$-optimal graphs for low dimension~$n$ is done by performing a graph search. The results from Section~\ref{sec:realizability} provide the conditions for this search. They also give lower bounds on~$\alpha(n, t)$. Theorem~\ref{thm:interpolation} and Theorem~\ref{thm:maximal-almost-equiangular-sets} provide an upper bound of~$\alpha(n, t) \leq 2(n+1)$  for~$t \in [-1, 0]$, which is only attained at~$t = -1/n$. There is a global lower bound of~$\alpha(n, t) \geq 4$, given by the disjoint union of two edges.

Perform the graph search as follows. Let~$t_{k,i}$ be the~$i$th root of the polynomial~\eqref{eq:MoserSpindlePolynomial} for fixed~$k$. Given~$t$ and~$2 \leq k \leq n$, list all graphs~$G$ of a given order that do not contain a subgraph isomorphic to:
\begin{itemize}
    \item an anti-triangle;
    \item $\comp{n + 2}$;
    \item an~$n$-rhombus;
    \item a~$k$-rhombus if~$t \leq -1/k$;
    \item $\comp{k + 1}$ if~$t < -1/k$;
    \item $\comp{n + 1}$ if~$t > -1 / n$;
    \item an extended $(n - 1)$-rhombus if~$t = -1/n$;
    \item $\MS{k}$ if~$t < t_{k, 1}$;
    \item $\MS{n - 1}$ if~$t > t_{n - 1, 2}$.
\end{itemize}

The search is implemented in SageMath in a script in the supplement to this paper. The code is a modified version of the code used in~\cite{Balko2020Almost-EquidistantSets}. Given a dimension~$n$, all anti-triangle free graphs of cardinality at most~$2(n+1)$ not containing a~$K_{n+2}$ are generated. A second script reduces the size of these sets greatly by only taking the minimal anti-triangle-free graphs. Finally, each graph is searched for the above list of subgraphs.  The results below are summarized in Figure~\ref{fig:dim2and3}.
\medbreak

\noindent
{\sc Dimension~$2$.}
The~$3$-point bound for~$t\leq 0$ proves a global upper bound~$\alpha(2, t) \leq 6$, which is only achieved at~$t = -1 / 2$ by the double triangle. A graph search on order~$6$ graphs that are anti-triangle-free and do not contain~$\comp{4}$ or a~$2$-rhombus shows that this is the only minimal order~6 construction on the circle.

The Moser Spindle~$\MS{1}$ is realizable for~$t = -(1 / 4)(1 \pm \sqrt{5})$ and has order~5. Its graph is a~5-cycle, which is the unique anti-triangle-free graph of order~5 containing no~$\comp{3}$. The inner product~$t = -(1/4)(1 + \sqrt{5}) = \cos(-4\pi / 5)$ corresponds to the pentagram and the inner product~$t = -(1/4)(1-\sqrt{5}) = \cos(-2\pi / 5)$ corresponds to the regular pentagon.

Every other anti-triangle-free graph of order~$5$ satisfying the constraints above contains the disjoint union of a triangle and an edge, which is only realizable at~$t = -1 / 2$. This shows that the~5-cycle is the unique optimal construction at~$t = -(1 / 4)(1 \pm \sqrt{5})$.

At every other inner product the maximum cardinality is~$4$, attained by two disjoint edges, which is the unique optimal construction of this order.
\medbreak

\noindent
{\sc Dimension~$3$.}
The~$3$-point bound for~$t \leq 0$ proves a global upper bound~$\alpha(3, t) \leq 8$. This is achieved by the double tetrahedron for~$t = -1 / 3$. The graph search shows that the double tetrahedron is the only minimal construction of order~$8$.

In the region~$t_{2,1} \leq t \leq t_{2,2}$, the Moser spindle~$\MS{2}$ is~$(3, t)$-realizable and of order~7. Excluding this subgraph from the graph search shows that it is the unique optimal construction for~$t_{2,1} \leq t < -1/3$ and~$-1/3 < t \leq t_{2,2}$.

For~$t \geq -1/2$, the double triangle is realizable. For~$t = -1/2$, it is the unique optimal construction. For~$t > -1/2$, the spindle~$\spindle{1}{2}$ is realizable and of order~6. Excluding these subgraphs from the graph search shows there are no other $(3, t)$-realizable graphs of order~$6$ with~$t > -1/2$. So for~$-1 / 2 < t < t_{2,1}$ and~$t_{2,2} < t \leq 0$, there are two optimal constructions of order~6.

The Moser spindle~$\MS{1}$ is the unique optimal construction for all~$-(1/4)(1 - \sqrt{5}) \leq t < -1/2$ and has order~$5$, as described above.


\section*{Acknowledgments}

We thank Nando Leijenhorst for help with the \texttt{ClusteredLowRankSolver.jl}
package and Willem de Muinck Keizer for helpful discussions.  Alexey Glazyrin
has pointed us to references~\cite{Bilyk2024OptimizersSets,
Bilyk2023OptimalPotentials}.


\newcommand{\arXiv}[1]{arXiv:#1}
\newcommand{\doi}[1]{}

\bibliographystyle{amsinit}
\bibliography{ref3.bib}




\end{document}